\documentclass[12pt,a4paper]{amsart}
\usepackage{graphicx,latexsym,amsfonts,amsmath,amssymb,rotating,txfonts,mathrsfs,enumerate, mathtools}
\usepackage[utf8]{inputenc}
\usepackage{epic}
\usepackage{curves}
\usepackage{pdfsync}
\usepackage{ytableau}
\usepackage{tikz}
\usepackage{calrsfs}
\usepackage{enumitem,kantlipsum}
\usepackage[justification=centering]{caption}
\usepackage{subcaption}

\input xy
\xyoption{all}

\newcommand{\Spec}{{\rm Spec}}

\newcommand{\Hom}{ \,{\rm Hom} \,}

\newcommand{\im}{ \,{\rm Im} \,}

{\bf}{\rm}
\newtheorem{theorem}{Theorem}[section]
\newtheorem*{theorem*}{Theorem}
\newtheorem{proposition}[theorem]{Proposition}
\newtheorem{corollary}[theorem]{Corollary}
\newtheorem{lemma}[theorem]{Lemma}
\newtheorem{definition}[theorem]{Definition}
\newtheorem{remark}[theorem]{Remark}
\newtheorem{conjecture}[theorem]{Conjecture}
\newtheorem{example}[theorem]{Example}
{\bf}{\it}

\newcommand{\RR}{{\mathbb R }}
\newcommand{\CC}{{\mathbb C }}

\newcommand{\ZZ}{{\mathbb Z }}
\newcommand{\PP}{ {\mathbb P }}
\newcommand{\QQ}{{\mathbb Q }}

\newcommand{\ff}{{\mathbf f }}

\newcommand{\calc}{\mathcal{C}}

\newcommand{\calo}{\mathcal{O}}

\newcommand{\calp}{\mathcal{P}}

\newcommand{\reg}{\mathrm{reg}}

\newcommand{\mon}{\mathrm{Mon}}

\newcommand{\Span}{\mathrm{Span}}

\newcommand{\boxx}{\mathrm{box}}

\newcommand{\symdot}{\mathrm{Sym}^{\le k}\CC^n}

\newcommand{\grass}{\mathrm{Grass}}

\newcommand{\flag}{\mathrm{Flag}}
\newcommand{\diff}{\mathrm{Diff}}

\newcommand{\bw}{\mathbf{w}}

\newcommand{\bff}{\mathbf{f}}
\newcommand{\bi}{\mathbf{i}}

\newcommand{\bv}{\mathbf{v}}

\newcommand{\bk}{\mathbf{k}}

\newcommand{\bV}{\mathbf{V}}
\newcommand{\jetk}[2]{J_{k}({#1},{#2})}

\newcommand{\jetreg}[2]{J_{k}^{\mathrm{reg}}({#1},{#2})}

\newcommand{\jetnondeg}[2]{J_{k}^{\mathrm{nondeg}}({#1},{#2})}

\newcommand{\GL}{\mathrm{GL}}

\newcommand{\sym}{\mathrm{Sym}}

\newcommand{\Hilb}{\mathrm{Hilb}}
\newcommand{\GHilb}{\mathrm{GHilb}}
\newcommand{\CHilb}{\mathrm{CHilb}}
\newcommand{\NHilb}{\mathrm{NHilb}}
\newcommand{\THilb}{\mathrm{THilb}}

\newcommand{\Quot}{\mathrm{Quot}}
\newcommand{\CQuot}{\mathrm{CQuot}}

\newcommand\Curv{\text{Curv}}

\newcommand\Soc{\mathrm{soc}}
\newcommand{\Alg}{ \mathrm {Alg} }
 
\newcommand\Dim{\mathrm{dim}}

\newcommand\RL{\mathrm {RL}}

\def\a{\alpha}
\def\b{\beta}
\def\g{\gamma}
\def\d{\delta}

\def\s{\sigma}

\def\vp{\varphi}

\usepackage{todonotes, tabularray,tikz-cd}
\newcommand\AAA{\mathbb{A}} 

\newcommand{\calT}{\mathcal{T}}

\newcommand{\calP}{\mathcal{P}}
\newcommand\Symm{\mathrm{Sym}}

\title{Fixed point distribution on Hilbert scheme of points} 

\author{Gergely B\'erczi}
\address{Department of Mathematics, Aarhus University}
\email{gergely.berczi@math.au.dk}
\author{Jonas M. Svendsen}
\address{Department of Mathematics, Aarhus University}
\email{svendsen@math.au.dk}

\begin{document}

\begin{abstract}
Let $\mathbf{k}$ be a closed field of characteristic zero. We prove that all monomial ideals sit in the curvilinear component of the Hilbert scheme of points of the affine space $\mathbb{A}_{\mathbf{k}}^n$, answering a long-standing question about the distribution of torus-fixed points among punctual components. This result confirms that the punctual Hilbert scheme is connected, a property previously established only for the full Hilbert scheme in 1966 by Hartshorne.  
\end{abstract}

\maketitle

\section{Introduction}\label{sec:intro}

Let $\bk$ be an algebraically closed field of characteristic zero, and let $\Hilb^k(\AAA_{\bk}^n)$ denote the Hilbert scheme of $k$ points on the affine space $\AAA_{\bk}^n$ parametrising subschemes of length $k$ of $\AAA_{\bk}^n$, alternatively, ideals of colength $k$ in $\bk[x_1,\ldots, x_n]$. 
We will set $\bk=\CC$ for the rest of this paper, and refer to \S \ref{sec:finalremarks} for comments on the chosen field. The punctual Hilbert scheme $\Hilb^k_0(\CC^n)$ is formed by subschemes supported at the origin. This punctual part contains all monomial ideals, which are the torus fixed points with respect to the natural action of the maximal torus of $\GL(n)$ on $\Hilb^k(\CC^n)$. This paper addresses the following long-standing question.

\noindent \textbf{Problem \cite{AIM,joachimsurvey}} What is the distribution of torus fixed points among the components of $\Hilb^k_0(\CC^n)$? 

The Hilbert scheme of points on $\CC^2$ is a nonsingular variety of dimension $2n$ (Fogarty \cite{fogarty}), and its punctual part is irreducible due to a theorem of Briancon \cite{briancon}, hence all fixed points sit in this unique component. For $n>2$ however, $\Hilb^k(\CC^n)$ presents a mixture of pathological and unknown behaviour, which Vakil characterised as a geometric example of Murphy's laws \cite{vakil}. For $k$ sufficiently large, neither $\Hilb^k(\CC^n)$ nor its punctual part $\Hilb^k_0(\CC^n)$ is irreducible, and the description of their components, singularities and deformation theory is out of reach at the moment \cite{joachimsurvey}. The importance of this moduli space in geometry and physics is undisputed, and its intersection theory plays crucial role in several classical problems of enumerative geometry and mathematical physics \cite{nakajima}. Nevertheless, certain fundamental characteristics continue to elude our understanding, see \cite{joachimsurvey} for a recent survey of open problems. 

There is a distinguished component of $\Hilb^k_0(\CC^n)$, that seems to play a central role in the topology and geometry of the Hilbert scheme of points: the curvilinear locus is the set of those length $k$ subschemes on $\CC^n$ which arise when $k$ points come together at the origin along a smooth curve-jet, and the curvilinear component is the closure of this locus: 
\[\CHilb^k_0(\CC^n)=\overline{\{I \subset \mathfrak{m}:\mathfrak{m}/I \simeq t\CC[t]/t^{k}\}}\]
where $\mathfrak{m}=(x_1,\ldots, x_n)$ is the maximal ideal at the origin. This is a small component of dimension $(n-1)(k-1)$, and in fact, conjecturally, it is the lowest dimensional punctual component on threefolds \cite{joachimsurvey,satriano}. In contrast, we know that there exist big, high-dimensional Iarrobino-type punctual components, typically isomorphic to Grassmannians of subspaces of $\mathfrak{m}^s/\mathfrak{m}^{s+1}$ containing many torus fixed-points. Given this perspective, our main theorem is rather unexpected. 

\begin{theorem}\label{main1} Every monomial ideal in $\Hilb^k(\CC^n)$ is contained in the curvilinear component $\CHilb^k_0(\CC^n)$.
\end{theorem} 

There are only very few general results on the topology of $\Hilb^k(\CC^n)$. Hartshorne \cite{hartshorne} proved that  $\Hilb^k(\CC^n)$ is connected and Horrocks \cite{horrocks} showed that $\Hilb^k(\PP^n)$ is simply connected for all $n,k$. Since all irreducible punctual components necessarily contain torus fixed points, Theorem \ref{main1} immediately implies the following fundamental result.

\begin{theorem}
 The punctual Hilbert scheme $\Hilb^k_0(\CC^n)$ is connected for all $k,n$.
\end{theorem}

Theorem \ref{main1} aligns with recent developments concerning the cohomology and intersection theory of the Hilbert scheme of points.

In \cite{totaro1} Hoyois, Jelisiejew, Nardin, Totaro and Yakerson study $\Hilb^k(\AAA^\infty)$ from a (stable) $\AAA^1$-homotopical viewpoint and show that $\Hilb^k(\AAA^\infty)$ is $\AAA^1$-equivalent to the Porteous locus 
\[\mathrm{Por}^k(\AAA^\infty)=\{I \subset \CC[x_1,\ldots, x_n]: \mathfrak{m}^2 \subset I \subset \mathfrak{m}\} \simeq \grass_k(\AAA^\infty)\]
in $\CHilb^k(\AAA^\infty)$ consisting of all ideals of colength $k$ sitting between the maximal ideal $\mathfrak m=(x_1,\ldots, x_n)$ and $\mathfrak{m}^2$, which is, in fact, the Grassmannian of $(k-1)$-planes in $\AAA^\infty$. 

In \cite{bercziG&T,berczitau2,berczitau3,berczitau4} the first author develops a new approach to tautological intersection theory of $\Hilb^k(\CC^n)$. The main insight is that tautological integrals on (the smoothable component of the) Hilbert scheme can be reduced to integrals over its punctual  curvilinear component $\CHilb^k_0(\CC^n)$. Then the results of \cite{bercziG&T} can be used where we develop an iterated residue formula for integrals over $\CHilb^k_0(\CC^n)$. We prove a residue vanishing theorem which shows that the fixed point contribution only comes from the Porteous locus defined above.   

Theorem \ref{main1} combined with the Atiyah-Bott localisation principle is yet another result that strongly supports our view, although it does not imply in any way, that the curvilinear component $\CHilb^k_0(\CC^n)$ must contain all topological information about the smoothable component of $\Hilb^k(\CC^n)$ in general. 

Our proof of Theorem \ref{main1} has grown out from \cite{bsz,bkcoh} and \cite{berczithom}, where a new non-reductive GIT model for the curvilinear Hilbert scheme was introduced, and toric resolutions of the Hilbert scheme were studied. In its final version, presented in this paper, the proof can be formulated without using the non-reductive GIT machinery, and use only the test curve model \cite{bsz,bercziG&T} of curvilinear Hilbert schemes. The crucial steps of the proof can be summarized as follows. (i) Using the model we prove that a special class of monomial ideals, which we call T-monomial ideals, sit in $\CHilb^k_0(\CC^n)$ (ii) We associate to any monomial ideal $I \in \Hilb^k_0(\CC^n)$ a T-monomial ideal $I^+ \in \CHilb^{K}_0(\CC^{N})$ for some $K\ge k$ and $N\ge n$ which we call the T-extension. (iii) Finally, we prove a descent property: if the $T$-extension $I^+ \in \CHilb^{K}_0(\CC^{N})$ sits in the curvilinear component, then $I \in \CHilb^{k}_0(\CC^n)$ also sits there. We emphasize that in this argument the pivotal steps involve both modifying the algebra using new coordinate directions (we call these socle extensions, which correspond to collide a point on $\CC^n$ from a new direction), and increasing the dimension of the ambient space. Our previous attempts and several years' work using the test curve model in a fixed dimension remained unsuccessful.

Finally, our technique is applicable to the study of fixed-point distribution on Quot schemes, which are higher rank analogs of Hilbert scheme of points \cite{joachimquot}. With the shorthand notation $S=\CC[x_1,\ldots, x_n]$, the Quot scheme parametrises $k$-dimensional quotients of the free $S$-module of rank $r$:  
\[\Quot_r^k(\CC^n)=\{J \subset S^{\oplus r}: \dim(S^{\oplus r}/J)=k\}.\]
This is naturally endowed with a $T^r \times \GL(n)$-action induced from the $\GL(n)$ action on $\CC^n$ and the rescaling action on the $r$ components. The fixed points for the action of the torus $T^r \times T^n$ correspond to monomial quotients of the form $S/\mon_1 \oplus \ldots \oplus S/\mon_r$ where $\mon_1,\ldots, \mon_r$ are monomial ideals of $S$. The curvilinear component is the closure of the curvilinear locus:
\[\CQuot^k_r(\CC^n)=\overline{\{J \subset S^{\oplus r}: S^{\oplus r}/J \simeq \CC[t]/t^k\}}.\]
In \S \ref{sec:finalremarks} we prove
\begin{theorem}
All torus fixed points of $\Quot_r^k(\CC^n)$ sit in the curvilinear component $\CQuot^k_r(\CC^n)$. Hence the punctual Quot scheme is connected. 
\end{theorem}

\textbf{Acknowledgments}. We are indebted to Joachim Jelisiejew and Damian Brotbek for useful discussions.




\section{Components of the Hilbert scheme of points}\label{sec:curvilinear}
Let $X$ be a nonsingular irreducible complex variety of dimension $n$ and let  
\[\Hilb^k(X)=\{\xi \subset X:\dim(\xi)=0,\mathrm{length}(\xi)=\dim H^0(\xi,\calo_\xi)=k\}\]
denote the Hilbert scheme of $k$ points on $X$ parametrizing all length-$k$ subschemes of $X$.
Despite intense study in the last 40 years, not much is known about the geometry and topology of $\Hilb^k(X)$ when $\dim(X)>2$. This section tries to hint the pathological behaviour and complexity this moduli space. We give a short overview of the structure of $\Hilb^k(X)$ referring to the Bellis Hilbertis in Figure \ref{figure:bellishilbertis}---a beautiful illustration introduced by Jelisiejew \cite{joachimthesis}---which summarizes its sophisticated structure as follows. For a comprehensive collection of open questions about the geometry of Hilbert scheme of points we refer to \cite{joachimsurvey}.

\begin{figure}[tb]
\centering
\includegraphics[width=80mm]{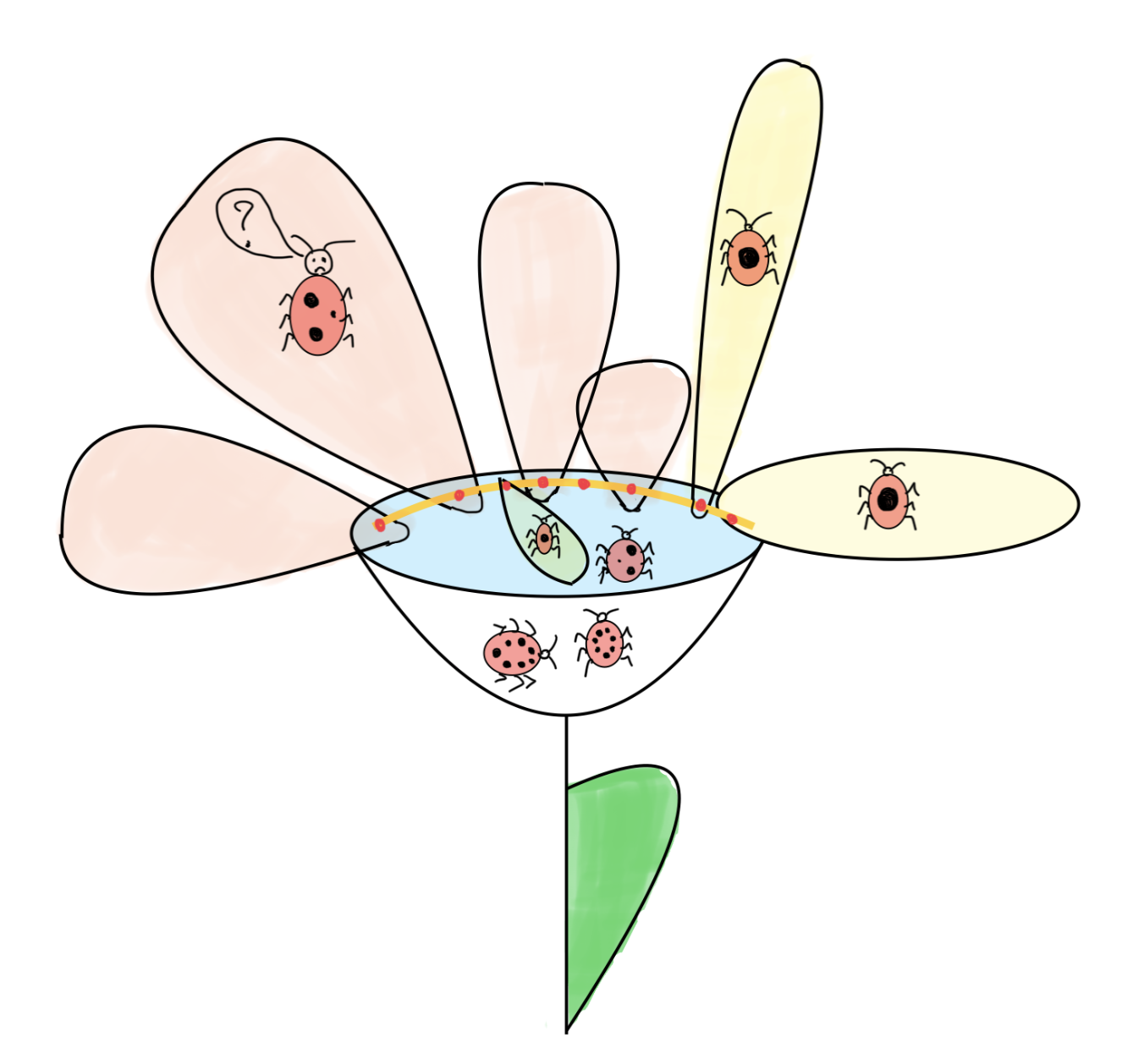}
\caption{The refined Bellis Hilbertis. \\ The stem is the smoothable component where a generic ladybird has $k$ dots. These dots collide on the blue boundary. The petals stand for the non-smoothable components of the Hilbert scheme. The yellow petals are the components of the punctual Hilbert scheme, where ladybirds have only one fat dot. Some of them are components of the Hilbert scheme too (like the Iarrobino-type components), others are smoothable (sitting in the blue stem boundary). The red petals are not punctual, and we don't know too much about them. The petals have different sizes: some of them are much bigger than the stem itself. The dark yellow line in the blue stem stands for the curvilinear component: it is a very small component intersecting all other punctual and non-punctual components, containing all torus fixed points indicated by red dots.} 
\label{figure:bellishilbertis}
\end{figure}

\begin{enumerate}[wide, labelwidth=!, labelindent=0pt]
    \item The main (also called smoothable or geometric) component of $\Hilb^k(X)$ 
    \[\GHilb^k(X)=\overline{\{p_1 \sqcup \ldots \sqcup p_i: p_i \neq p_j\}}\] 
    is the closure of the open set consisting of reduced subschemes, whose support consists of $k$ different points on $X$. Its dimension is $k\dim(X)$. In \cite{ekedahl} the good component is constructed as a certain blow-up along an ideal of $S^kX$, providing an extension of the classical result of Haiman \cite{haiman} on surfaces. 

    In Figure \ref{figure:bellishilbertis} the main component is the stem, and for small $k$ it is the only component, there are no petals. Generic points on the main component (ladybirds) correspond to $k$-tuples of points on $X$. All subschemes in $\GHilb^k(X)$ are limits of reduced subschemes, so they are smoothable.
    \item For $p\in X$ let 
\[\Hilb^k_p(X)=\{ \xi \in \Hilb^k(X): \mathrm{supp}(\xi)=p\}\]
denote the punctual Hilbert scheme consisting of subschemes supported at $p$. If $\rho: \Hilb^k(X) \to S^kX$, $\xi \mapsto \sum_{p\in X}\mathrm{length}(\calo_{\xi,p})p$ denotes the Hilbert-Chow morphism then $\Hilb^k_p(X)=\rho^{-1}(kp)$. The punctual Hilbert scheme is irreducible on surfaces (due to Briancon \cite{briancon}), but it is reducible when $\dim(X)>2$ and the number of points $k$ is sufficiently large, see \cite{iarrobino2,iarrobino}.
    \item Every other component (petal) intersects the smoothable one (see \cite{Ree95}), however, description of these intersections is currently out of reach.
    The first non-smoothable components were constructed by Iarrobino \cite{iarrobino2} for $n=3$ and $k\ge 8$. He called these \textit{very compressed algebras}: they are punctual (supported at one point) with dimension typically exceeding $k\dim(X)$. Very compressed algebras are typically constructed as  
\[\mathrm{VC}^{[k]}_p=\{I:\mathfrak{m}^s \subseteq I \subseteq \mathfrak{m}^{s+1},\dim_\CC \mathfrak{m}/I=k-1\}\]
where $\mathfrak{m} \subset \calo_p(X)$ is the maximal ideal and $s$ is the unique positive integer satisfying $\dim (\mathfrak{m}/\mathfrak{m}^s) <k <\dim (\mathfrak{m}/\mathfrak{m}^{s+1})$. In fact, a subspace of $V=\mathfrak{m}^s/\mathfrak{m}^{s+1}$ of dimension $k-\dim (\mathfrak{m}/\mathfrak{m}^s)-1$ determines the ideal uniquely and hence
\[\mathrm{VC}^{[k]}_p=\grass_{k-\dim (\mathfrak{m}/\mathfrak{m}^s)-1}(V).\]
In the other direction, Satriano and Staal \cite{satriano} recently found small punctual components, whose dimensions are smaller than the curvilinear component. These Satriano-Staal components are not smoothable. 
The existence of non-reduced components was recently proved by Jelisiejew \cite{joachiminventiones}.  
    \item For surfaces ($\Dim \,X=2$) the punctual geometric set $\GHilb^k_p(X)=\Hilb_p^{k}(X) \cap \GHilb^k(X)$
is irreducible and equal to the punctual Hilbert scheme $\Hilb_p^{k}(X)$ (see \cite{briancon}). However, for $n\ge 3$ $\GHilb^k_p(X)$ is typically not irreducible; its components are called smoothable components of $\Hilb_p^{k}(X)$. Again, description of the smoothable components is a hard problem and is unknown in general. The Iarrobino-type very compressed punctual components are clearly not smoothable: their dimension can be higher than $k\dim(X)$. Recently big, divisorial smoothable punctual components were found by Cartwright, Erman, Velasco, Viray \cite{Cartwright2009} (when $n\ge 4, k=8$) and Bertone, Cioffi, Roggero \cite{bertone} (when $n=7,k=16$).
Regarding small smoothable components, it is a conjecture that the curvilinear components (defined below) is the smallest punctual smoothable component. 
\item A subscheme $\xi \in \Hilb_p^k(X)$ is called curvilinear if $\xi$ is contained in some smooth curve $\calc \subset X$, or equivalently, $\calo_\xi \simeq \CC[z]/z^{k}$.
The curvilinear locus at $p\in X$ is the set of curvilinear subschemes supported at $p$: 
\begin{equation}\nonumber 
\mathrm{Curv}^k_p(X)=\{\xi \in \Hilb^k_p(X): \xi \subset \mathcal{C}_p \text{ for some smooth curve } \mathcal{C} \subset X\}
=\{\xi:\calo_\xi \simeq \CC[z]/z^{k}\}
\end{equation}
Note that any curvilinear subscheme contains only one subscheme for any given smaller length and any small deformation of a curvilinear subscheme is again locally curvilinear.

For surfaces ($\Dim \,X=2$) $\mathrm{Curv}^k_p(X)$ is an irreducible quasi-projective variety of dimension $k-1$ which is an open dense subset in $\Hilb^k_p(X)$ and therefore its closure is the full punctual Hilbert scheme at $p$. However, if $n\ge 3$, the punctual Hilbert scheme $\Hilb^k_p(X)$ is not necessarily irreducible or reduced, and it has smoothable and non-smoothable components. The closure of the curvilinear locus is one of its smoothable irreducible components as the following lemma shows. 
\begin{lemma} $\CHilb^k_p(X)=\overline{\mathrm{Curv}^k_p(X)}$ is an irreducible component of the punctual Hilbert scheme $\Hilb^k_p(X)$ of dimension $(n-1)(k-1)$; we call this the curvilinear component. 
\end{lemma}

\begin{proof}
To study the punctual Hilbert scheme we can take $X=\CC^n$ and $p=\{0\}$ the origin. Note that $\xi \in \Hilb^{k}_0(\CC^n)$ is not curvilinear if and only if $\calo_\xi$ does not contain elements of degree $k-1$, that is, after fixing some local coordinates $x_1,\ldots, x_n$ of $\CC^n$ at the origin we have
\[\calo_\xi \simeq \CC[x_1,\ldots, x_n]/I \text{ for some } I\supseteq (x_1,\ldots, x_n)^{k-1}.\]
This is a closed condition and therefore curvilinear subschemes can't be approximated by non-curvilinear subschemes in $\Hilb^{k}_0(\CC^n)$. The dimension will follow from the description of it as a non-reductive quotient in the next subsection.  
\end{proof}
\end{enumerate}

\begin{remark}\label{remark:embed}
Recall that the defining ideal $I_\xi$ of any subscheme $\xi \in \Hilb^{k+1}_0(\CC^n)$ is a codimension $k$ subspace in the maximal ideal $\mathfrak{m}=(x_1,\ldots, x_n)$. The dual of this is a $k$-dimensional subspace $S_\xi$ in $\mathfrak{m}^*\simeq \symdot$ giving us a natural embedding $\varphi: \Hilb^{k+1}_p(X)=\Hilb^{k+1}_0(\CC^n) \hookrightarrow \grass_k(\symdot)$. The test curve model, which we introduce in the next section, provides explicit  parametrisation of this embedding using an algebraic model coming from global singularity theory.  
\end{remark}

\section{The test curve model of curvilinear Hilbert schemes}

In this section we recall an algebraic model for the curvilinear Hilbert scheme. This model was originally developed in \cite{gaffney,bsz} as a compactification of contact singularity classes of Morin type, and localisation theory on this model resulted in new formulas for Thom polynomials in \cite{bsz}. The model was later succesfully applied in \cite{bkGGL,bercziG&T,berczitau2,berczitau3,berczithom}.

\subsection{Jets of holomorphic maps}\label{subsubsec:holmaps} For $u,v$ positive integers we introduce the notation  
\[J_k(u,v)=\{f: (\CC^u,0) \to (\CC^v,0) \text{ holomorphic germ}\}/\sim \]
where 
\[f\sim g \Leftrightarrow f^{(j)}(0)=g^{(j)}(0) \text{ for } j=1,\ldots ,k.\]
In short, $J_k(u,v)$ is the vector space of $k$-jets of germs $f:(\CC^u,0) \to (\CC^v,0)$, and we simply refer to elements of
$J_k(u,v)$ as map-jets. Choosing coordinates on $\CC^u$ and $\CC^v$ a $k$-jet $f \in J_k(u,v)$ can be identified with the set of derivatives at the
origin, that is the vector $(f'(0),f''(0),\ldots, f^{(k)}(0))$, where
$f^{(j)}(0)\in \mathrm{Hom}(\mathrm{Sym}^j\CC^u,\CC^v)$. This way we
get the equality
\begin{equation}\label{identification}
J_k(u,v) \simeq J_k(u,1) \otimes \CC^v \simeq \oplus_{j=1}^k\mathrm{Hom}(\mathrm{Sym}^j\CC^u,\CC^v).
\end{equation}
Hence
$\dim J_k(u,v) =v \binom{u+k}{k}-v$. 

The composition map 
\begin{equation}
  \label{comp}
\jetk uv \times \jetk vw \to \jetk uw,\;\;  (\Psi_1,\Psi_2)\mapsto
\Psi_2\circ\Psi_1 \mbox{ modulo terms of degree $>k$ }.
\end{equation}
is defined naturally by composing map-jets via substitution followed by elimination of terms
of degree at least $k+1$.
When $k=1$, $J_1(u,v)$ may be identified with $u$-by-$v$ matrices,
and \eqref{comp} reduces to multiplication of matrices.

Elements
of $J_k(1,n)$ are $k$-jet of curves $\gamma: (\CC,0) \to (\CC^n,0)$. We introduce the notation $\jetreg 1n$ for the set
of regular curves:
\[\jetreg 1n=\left\{\g \in \jetk 1n; \g'(0)\neq 0 \right\}\]
consisting of curves with non-vanishing first derivative.
Note that $\jetreg uu$ with the composition map \eqref{comp} has a natural group structure and we will use the notation
\[\mathrm{Diff}_k(u)=\jetreg uu\]
for this reparametrisation group. We note that $\mathrm{Diff}_k(u)$ is not reductive, and it played a central role in the development of non-reductive geometric invariant theory \cite{bkcoh}.

Eliminating the terms of degree $k$ defines the natural algebra homomorphism
$J_k(u,1) \twoheadrightarrow J_{k-1}(u,1)$, and hence the sequence of homomorphisms
$J_k(u,1) \twoheadrightarrow J_{k-1}(u,1)
\twoheadrightarrow \ldots \twoheadrightarrow J_1(u,1)$. This  induces
an increasing filtration on $J_k(u,1)^*$:
\begin{equation}\label{filtration} 
J_1(u,1)^* \subset J_2(u,1)^* \subset \ldots \subset J_k(u,1)^*
\end{equation}
Elements of $J_i(u,1)^*$ can be interpreted as differential operators on $\CC^u$ of degree at most
$i$, and hence 
\begin{equation}\label{eq:dual}
  J_k(u,1)^* \cong \sym^{\le k}\CC^u \overset{\mathrm{def}}=
\oplus_{l=1}^k \sym^l\CC^u,
\end{equation}
where $\sym^l$ stands for the symmetric tensor product and the
isomorphism is that of filtered $\GL(n)$-modules.

\subsection{The test curve model}
The idea of the model is simple: let $\xi \in \Curv^{k+1}_0(\CC^n)$ be a curvilinear subscheme supported at the origin on $\CC^n$. This means that $\calo_\xi \simeq \CC[t]/t^{k+1}$ and hence $\xi$ is (scheme theoretically) contained in a smooth (germ of a) curve  $\mathcal{C}_0$ in $\CC^n$:
\[\xi \subset \mathcal{C}_0 \subset \CC^n.\]
Let $f_{\xi}:(\CC,0)\to (\CC^n,0)$ be a $k$-jet which  parametrises $\mathcal{C}_0$. 
This $f_{\xi}\in J^\reg_k(1,n)$ is determined only up to polynomial reparametrisations of the source space $\phi: (\CC,0)\to (\CC,0)$ and therefore 
\begin{lemma}\label{lemma:quotient} The punctual curvilinear locus $\Curv^{[k+1]}_0(\CC^n)$ is equal (as a set) to the set of $k$-jet of regular germs at the origin, modulo polynomial reparametrisations: 
\begin{equation}\nonumber
\Curv^{k+1}_0(\CC^n)=\{\text{regular }k\text{-jets } (\CC,0)\to (\CC^n,0)\}/\{\text{regular }k\text{-jets } (\CC,0)\to (\CC,0)\}.
\end{equation}
\end{lemma}
We introduce the shorthand notation $\diff_k(1)=J_k^\reg(1,1)$ for the group of polynomial reparametrisations of the complex line at the origin. Lemma \ref{lemma:quotient} then can be rephrased as $\Curv^{k+1}_0(\CC^n)=J_k^\reg(1,n)/\diff_k(1)$.
The action of $\diff_k(1)$ on $\jetreg 1n$ is defined in \eqref{comp}, and can be written in coordinates as follows. Let 
\[f_{\xi}(z)=z f'(0)+\frac{z^2}{2!}f''(0)+\ldots +\frac{z^k}{k!}f^{(k)}(0) \in \jetreg 1n\]
be the $k$-jet of a germ at the origin (i.e no constant term) in $\CC^n$ with $f^{(i)}\in \CC^n$ such that $f' \neq 0$ and let  
\[\varphi(z)=\alpha_1z+\alpha_2z^2+\ldots +\alpha_k z^k \in \jetreg 11\]
with $\alpha_i\in \CC, \alpha_1\neq 0$. Then 
 \[f \circ\varphi(z)
=(f'(0)\alpha_1)z+(f'(0)\alpha_2+
\frac{f''(0)}{2!}\alpha_1^2)z^2+\ldots
+\left(\sum_{i_1+\ldots +i_l=k}
\frac{f^{(l)}(0)}{l!}\alpha_{i_1}\ldots \alpha_{i_l}\right)z^k=\]
\begin{equation}\label{jetdiffmatrix}
=(f'(0),\ldots, f^{(k)}(0)/k!)\cdot 
\left(
\begin{array}{ccccc}
\alpha_1 & \alpha_2   & \alpha_3          & \ldots & \alpha_k \\
0        & \alpha_1^2 & 2\alpha_1\alpha_2 & \ldots & 2\alpha_1\alpha_{k-1}+\ldots \\
0        & 0          & \alpha_1^3        & \ldots & 3\alpha_1^2\alpha_ {k-2}+ \ldots \\
0        & 0          & 0                 & \ldots & \cdot \\
\cdot    & \cdot   & \cdot    & \ldots & \alpha_1^k
\end{array}
 \right)
 \end{equation}
where the $(i,j)$ entry is $p_{i,j}(\bar{\alpha})=\sum_{a_1+a_2+\ldots +a_i=j}\alpha_{a_1}\alpha_{a_2} \ldots \alpha_{a_i}.$

\begin{remark}\label{naturalembedding}
The linearisation of the action of $\diff_k(1)$ on $\jetreg 1n$ can be represented as a matrix multiplication, as shown in equation \eqref{jetdiffmatrix}. Under this representation $\diff_k(1)$ is an upper triangular matrix group in $\GL(n)$, which is not reductive; hence Mumford's reductive GIT cannot be applied to study the geometry of the quotient $\jetreg 1n/\diff_k(1)$. 

Note that the matrix group is parameterized along its first row, with the free parameters $\alpha_1, \ldots, \alpha_k$. The remaining entries are determined by certain (weighted homogeneous) polynomials in these free parameters. The group is the $\CC^*$ extension of its maximal unipotent radical:
\[\diff_k(1) = U \rtimes \CC^*,
\]
where $U$ is the subgroup obtained by substituting $\alpha_1 = 1$, and the diagonal $\CC^*$ acts on the Lie algebra $\mathrm{Lie}(U)$ with weights $0, 1, \ldots, n-1$.

For a broader context and analysis of actions of groups with similar characteristics see B\'erczi and Kirwan \cite{bkcoh} and B\'erczi, Doran, Hawes, and Kirwan \cite{bdhk, bdhk2}. 
\end{remark}

Fix an integer $N\ge 1$ and define
\[\Theta_k=\left\{\Psi\in J_k(n,N):\exists \g \in \jetreg 1n: \Psi \circ \g=0
\right\}.\]
$\Theta_k$ is the set of those $k$-jets of germs on $\CC^n$ at the origin which vanish on some regular curve. We call $\g$ a test
curve of $\Theta$ if $\Psi \circ \gamma=0$ holds. By definition, $\Theta_k$ is the image
of the closed subvariety of $\jetk nN \times \jetreg 1n$ defined by
the algebraic equations $\Psi \circ \g=0$, under the projection to
the first factor.  

\begin{remark}
The subset $\Theta_k$ in the jet space $J_k(n,N)$ is the closure of an important class of singularities known as Morin singularities. The equivariant dual of $\Theta_k$ in $J_k(n,N)$ is the Thom polynomial of Morin singularities, see B\'erczi and Szenes \cite{bsz} and Feh\'er and Rim\'anyi \cite{rf} for details. 
\end{remark}

The key observation is the following: test curves of germs are not unique; if $\g$ is a test
curve of $\Psi \in \Theta_k$, and $\vp \in \diff_k(1)$ is a 
holomorphic reparametrisation of $\CC$, then $\g \circ \vp$ is,
again, a test curve of $\Psi$: 
\begin{displaymath}
\label{basicidea}
\xymatrix{
  \CC \ar[r]^\vp & \CC \ar[r]^\g & \CC^n \ar[r]^{\Psi} & \CC^N}
\end{displaymath}
\[\Psi \circ \g=0\ \ \Rightarrow \ \ \ \Psi \circ (\g \circ \vp)=0\]

In \cite{bsz} we prove that we get all test curves of $\Psi$ in this way if the linear part of $\Psi$ has
$1$-dimensional kernel, which is an open property in $\theta_k$.  Before stating this in Theorem 
\ref{embedgrass} below, we illustrate the equation $\Psi \circ
\g=0$ for small order.  Let
$\g=(\g',\g'',\ldots, \g^{(k)})\in \jetreg 1n$ and
$\Psi=(\Psi',\Psi'',\ldots, \Psi^{(k)})\in \jetk nN$ be the
$k$-jets of the test curve $\g$ and the map $\Psi$ respectively. Using the chain rule and the notation $v_i=\g^{(i)}/i!$, the equation $\Psi \circ \g=0$ reads
as follows for $k=4$:
\begin{eqnarray} \label{eqn4}
& \Psi'(v_1)=0 \\ \nonumber & \Psi'(v_2)+\Psi''(v_1,v_1)=0 \\
\nonumber
& \Psi'(v_3)+2\Psi''(v_1,v_2)+\Psi'''(v_1,v_1,v_1)=0 \\
&
\Psi'(v_4)+2\Psi''(v_1,v_3)+\Psi''(v_2,v_2)+
3\Psi'''(v_1,v_1,v_2)+\Psi''''(v_1,v_1,v_1,v_1)=0
\nonumber
\end{eqnarray}


\begin{lemma}[Gaffney \cite{gaffney}, B\'erczi-Szenes \cite{bsz}]\label{explgp} Let
$\g=(\g',\g'',\ldots, \g^{(k)})\in \jetreg 1n$ and
$\Psi=(\Psi',\Psi'',\ldots, \Psi^{(k)})\in \jetk nN$ be $k$-jets.
Then substituting $v_i=\g^{(i)}/i!$, the equation $\Psi\circ \g=0$ is equivalent to
  the following system of $k$ linear equations with values in
  $\CC^N$:
\begin{equation}
  \label{modeleq}
\sum_{i_1+\ldots +i_s=m} \Psi^{(s)}(v_{i_1},\ldots, v_{i_s})=0,\quad m=1,2,\dots, k
\end{equation}
\end{lemma}
The equations \eqref{modeleq} are linear in $\Psi$, hence for a given $\g \in \jetreg 1n$ and $1\le i \le k$ the solution set of the first $i$ equations in \eqref{modeleq}
\begin{equation}\label{solutionspace}
\mathcal{S}^{i,N}_\g=\left\{\Psi \in \jetk nN, \Psi \circ \g=0 \text{ up to order } i \right\}\subset \jetk nN
\end{equation}
is a linear subspace of codimension $iN$. Hence $\mathcal{S}^{i,N}_\g$ is a point of $\grass_{\mathrm{codim}=iN}(J_k(n,N))$, whose orthogonal, $(\mathcal{S}^{i,N}_{\g})^\perp$, is an $iN$-dimensional subspace of $J_k(n,N)^*$. The $N$ vanishing coordinates of $\Psi \circ \gamma$ represent this subspace as 
\[(\mathcal{S}^{i,N}_{\g})^\perp=(\mathcal{S}^{i,1}_{\g})^\perp \otimes \CC^N,\]
and this subspace is invariant under the reparametrisation of $\gamma$. 
When $N\ge n$ then
\[\tilde{\mathcal{S}}^{i,N}_{\g}=\{\Psi \in \mathcal{S}^{i,N}_{\g}: \dim \ker \Psi^1=1\}\]
is an open subset of the subspace $\mathcal{S}^{i,N}_{\g}$. In fact it is not hard to see that the complement $\tilde{\mathcal{S}}^{i,N}_{\g}\setminus \mathcal{S}^{i,N}_{\g}$where the kernel of $\Psi^1$ has dimension at least two is a closed subvariety of codimension $N-n+2$. The next theorem first appeared in \cite{bsz}, but to make the this section self-contained, we present the proof again.


\begin{theorem}[B\'erczi-Szenes \cite{bsz}]\label{embedgrass} The map
\[\phi: \jetreg 1n \rightarrow \grass_k(J_k(n,1)^*)\]
defined as  $\gamma  \mapsto (\mathcal{S}^{k,1}_\g)^\perp$
is $\diff_k(1)$-invariant and induces an injective map on the $\diff_k(1)$-orbits into the Grassmannian 
\[\phi^\grass: \jetreg 1n /\diff_k(1) \hookrightarrow \grass_k(J_k(n,1)^*).\]
Moreover, $\phi$ and $\phi^\grass$ are $\GL(n)$-equivariant with respect to the standard action of $\GL(n)$ on $\jetreg 1n \subset \Hom(\CC^k,\CC^n)$ and the induced action on $\grass_k(J_k(n,1)^*)$.
\end{theorem}

\begin{proof}
For the first part it is enough to prove that for $\Psi \in \Theta_k$ with $\dim \ker \Psi^1=1$ and $\gamma,\delta \in \jetreg 1n$
\[\Psi \circ \gamma=\Psi \circ \delta=0 \Leftrightarrow \exists
\Delta \in \jetreg 11 \text{ such that } \gamma=\delta
\circ \Delta.\]
We prove this statement by induction. Let $\gamma=v_1t+\dots +v_kt^k$ and
$\delta=w_1t+\dots+ w_kt^k$. Since $\dim \ker \Psi^1=1$, $v_1=\lambda w_1$, for some
$\lambda\neq0$. This proves the $k=1$ case. 

Suppose the statement is true for $k-1$. Then, using the appropriate
order-($k-1$) diffeomorphism, we can assume that $v_m=w_m$, $m=1\ldots
k-1.$ It is clear then from the explicit form \eqref{modeleq}
(cf. \eqref{eqn4}) of the equation
$\Psi\circ\gamma=0$, that  $\Psi^1(v_k)=\Psi^1(w_k)$, and hence
$w_k=v_k-\lambda v_1$ for some $\lambda\in\CC$. Then
$\gamma=\Delta\circ\delta$ for $\Delta=t+\lambda t^k$, and the proof
is complete.
\end{proof}

\begin{remark}\label{remark:orbit}
For a point $\gamma\in J_k^\reg(1,n)$ let $v_i=\frac{\g^{(i)}}{i!}\in \CC^n$ denote the normed $i$th derivative. Then from Lemma \ref{explgp} immediately follows that for $1\le i \le k$ (see \cite{bsz}):
\begin{equation}\label{sgamma}
(\mathcal{S}^{i,1}_\g)^\perp=\mathrm{Span}_\CC (v_1,v_2+v_1^2,\ldots, \sum_{i_1+\ldots +i_s=k}v_{i_1}\ldots v_{i_s})\subset \symdot
\end{equation} 
\end{remark}

Recall from \eqref{eq:dual} that $J_k(n,1)^*=\symdot$. Note that the image of $\phi$ and the image of $\varphi$ defined in Remark \ref{remark:embed} coincide in $\grass_k(\symdot)$:
\[\mathrm{im}(\phi)=\mathrm{im}(\varphi)\subset \grass_k(\symdot).\]
Hence their closure is the same, which is the first part of the next theorem. The second part follows from Remark \ref{remark:orbit} because $\phi$ is $\GL(n)$-equivariant.
\begin{theorem}[Test curve model for $\CHilb^{k+1}_0(\CC^n)$, \cite{bercziG&T}]\label{bszmodel}
\begin{enumerate}
\item For any $k,n$ we have 
\[\CHilb^{k+1}_0(\CC^n)=\overline{\mathrm{im}(\phi^\grass)} \subset \grass_k(\symdot).\]
\item Let $\{e_1,\ldots, e_n\}$ be a basis of $\CC^n$. For $k\le n$ the $\GL(n)$-orbit of 
\[p_{k,n}=\phi(e_1,\ldots, e_k)=\mathrm{Span}_\CC (e_1,e_2+e_1^2,\ldots, \sum_{i_1+\ldots +i_s=k}e_{i_1}\ldots e_{i_s})\] 
forms a dense subset of the image $\jetreg 1n$ and therefore 
\[\CHilb^{k+1}_0(\CC^n)=\overline{\mathrm{GL_n} \cdot p_{k,n}}.\] 
\end{enumerate}
\end{theorem}
The nested Hilbert scheme is defined as the set of flags of subschemes 
\[\NHilb^{k}(\CC^n)=\{(\xi_1 \subset \xi_2 \subset \ldots, \subset \xi_k): \xi_i \in \Hilb^i(\CC^n)\}\]
The curvilinear nested Hilbert scheme supported at the origin is  
\[\mathrm{CNHilb}^{k}_0(\CC^n)=\{(\xi_1 \subset \xi_2 \subset \ldots, \subset \xi_k): \xi_i \in \CHilb^i_0(\CC^n)\}.\]
The test curve model can be adapted for nested Hilbert schemes. For a complex vector space $V$ let 
\begin{equation}\label{def:flag}
\flag_{1,\ldots, k}(V)=\{(V_1 \subset \ldots V_{k}\subset V| \dim(V_i)=i\}
\end{equation}
denote the partial flag manifold of dimension vector $(1,2,\ldots, k)$. Let  
\[\phi^{\flag}: \jetreg 1n /\diff_k(1) \hookrightarrow \flag_\bk(J_k(n,1)^*).\]
be the embedding defined as $\gamma  \mapsto (\mathcal{S}^{1,1}_\g)^\perp \subset \ldots \subset (\mathcal{S}^{k,1}_\g)^\perp$.

\begin{corollary}[Test curve model for $\mathrm{CNHilb}^k_0(\CC^n)$, \cite{bercziG&T}]\label{bszmodelfiltered}
\begin{enumerate}
\item For any $k,n$  
\[\mathrm{CNHilb}^{k+1}_0(\CC^n)=\overline{\mathrm{im}(\phi^\flag)} \subset \flag_\bk(\symdot).\]
\item For $k\le n$ the $\GL(n)$-orbit of the reference flag
\[f_{k,n}=(p_{1,n} \subset p_{2,n} \subset \ldots \subset p_{k,n})\] 
forms a dense subset of the image $\jetreg 1n$ and therefore 
\[\mathrm{CNHilb}^{k+1}_0(\CC^n)=\overline{\mathrm{GL_n} \cdot f_{k,n}}.\] 
\end{enumerate}
\end{corollary}

\subsection{Fibration over the flag manifold}\label{subsubsec:blowingup} 
In this section we assume that $k\le n$ holds, and introduce a fibration which plays crucial role in the intersection theory of the Hilbert scheme of points. In \cite{bsz,bercziG&T,bkGGL,berczitau2} this fibration allows us rewriting equivariant localisation to iterated residue formulas, which show some unexpected residue vanishing properties in the referred works. In this paper we use this fibration to identify distinguished monomial ideals in each fiber, which we will later call $T$-monomial ideals. 

Let $\jetnondeg 1n \subset \jetreg 1n$ denote the Zariski open set of jets $(\g',\g'' \ldots, \g^{[k]})$ with $\g',\ldots, \g^{(k)}$ linearly independent. These correspond to regular $n \times k$ matrices in $\Hom(\CC^k,\CC^n)$, and they fibre over the partial flag manifold $\flag_{1,\ldots, k}(\CC^n)$ defined in \eqref{def:flag}:
\[\jetnondeg 1n=\Hom^\reg(\CC^k,\CC^n) \to \Hom^\reg(\CC^k,\CC^n)/B_k=\flag_{1,\ldots, k}(\CC^n)\]
where $B_k \subset \GL(k)$ is the upper Borel. Since $\jetreg 11 \subset B_k$ this induces a surjective fibration
\begin{equation}\label{proj}
\pi:\jetnondeg 1n/\diff_k(1) \to \flag_{1,\ldots, k}(\CC^n).
\end{equation}
which factors through $\phi^\grass$ and $\phi^\flag$
\begin{equation}\label{proj2}
\xymatrix{\jetnondeg 1n/\diff_k(1) \ar[r]^-{\phi^\flag} \ar@/^2pc/[rr]^-{\phi^\grass} \ar[rd]^\pi & \flag_{1,\ldots, k}(\symdot) \ar@{-->}[d] \ar[r] & \grass_k(\symdot) \ar@{-->}[ld]\\
& \flag_{1,\ldots, k}(\CC^n)& & } 
\end{equation}
Here the vertical rational map is induced by the projection $\symdot \to \CC^n$ and the image of $\phi^\flag$ sits in its domain. We will work over a fixed flag $\ff$, and introduce the following partial blow-ups.

\begin{definition}\label{def:widetilde}  Fix a basis $\{e_1,\ldots, e_n\}$ of $\CC^n$ and let $B_{k,n} \subset \GL(n)$ denote the parabolic subgroup which preserves the flag 
\[\mathbf{f}=(\mathrm{Span}(e_1)   \subset \mathrm{Span}(e_1,e_2) \subset \ldots \subset \mathrm{Span}(e_1,\ldots, e_k) \subset \CC^n).\] 
\begin{enumerate} 
\item We let  
\[\widehat{\CHilb}_0^{k+1}(\CC^n)=\GL(n) \times_{B_{k,n}} \overline{B_{k,n}\cdot p_{k,n}}.\] 
be the fiberwise compactification in $\grass_k(\symdot)$. It is a partial resolution of $\CHilb^{k+1}_0(\CC^n)$:
\[\rho: \widehat{\CHilb}^{k+1}_0(\CC^n)=\GL(n) \times_{B_{k,n}} \overline{B_{k,n}\cdot p_{k,n}} \to \overline{\GL(n)\cdot p_{k,n}}=\CHilb^{k+1}_0(\CC^n).\]
\item Similarly, let 
\[\widehat{\NHilb}^{k+1}_0(\CC^n)=\GL(n) \times_{B_{k,n}} \overline{B_{k,n}\cdot f_{k,n}}\] 
be the fiberwise compactification of the nested Hilbert scheme in $\flag_{1,\ldots, k}(\symdot)$, with the partial resolution map:
\[\rho: \widehat{\NHilb}^{k+1}_0(\CC^n)=\GL(n) \times_{B_{k,n}} \overline{B_{k,n}\cdot f_{k,n}} \to \overline{\GL(n)\cdot f_{k,n}}=\NHilb^{k+1}_0(\CC^n).\]
\item Finally, let $T_k \subset \GL(n)$ be the $k$-dimensional torus acting on $\CC_{[k]}=\Span(e_1,\ldots, e_k)$ diagonally. We define the toric Hilbert scheme as the fiberwise compactification in $\grass_k(\symdot)$ of the torus orbit:
\[\widehat{\THilb}^{k+1}_0(\CC^n)=\GL(n) \times_{B_{k,n}} \overline{T_k\cdot p_{k,n}}\] 
\end{enumerate}
\end{definition}

Note that these partial resolutions are fibrations over $\GL(n)/B_{k,n}=\flag_{1,\ldots, k}(\CC^n)$, establishing the $\GL(n)$-equivariant diagram 
\begin{equation}\label{diagram:res}
\xymatrix{\widehat{\CHilb}^{k+1}_0(\CC^n)=\GL(n) \times_{B_{k,n}} \overline{B_{k,n}\cdot p_{k,n}} \ar@{->>}[r]^-\rho \ar[d]^{\pi^\Hilb}  &  \CHilb^{k+1}_0(\CC^n)=\overline{\GL(n)\cdot p_{k,n}}\\
\flag_{1,\ldots, k}(\CC^n)}
\end{equation}

Using the test curve map 
\[\phi_k : J_k^\reg(1,n) \to \grass_k(\symdot) \subset \PP(\wedge^k (\oplus_{i=1}^k \Symm^i \CC^n))\]
\[(0\neq v_1,v_2, \ldots, v_k) \mapsto [v_1 \wedge (v_2+v_1^2) \wedge \cdots \sum_{i_1+\dots+i_s = k} v_{i_1}\cdots v_{i_s}]\]
the fibers over the base flag $\ff=(\mathrm{Span}(e_1)   \subset \mathrm{Span}(e_1,e_2) \subset \ldots \subset \mathrm{Span}(e_1,\ldots, e_k) \subset \CC^n)$
can be explicitely described as follows: 
\begin{align}
    \widehat{\CHilb}^{k+1}_\ff=& \overline{B_{k,n} p_{k,n}}=\overline{\left\{\phi_k(\bv): v_i=b_{ii}e_i+b_{i-1,i}e_{i-1}+\ldots +b_{1i}e_1 \right\}} \label{fiberchilb}\\
    \widehat{\THilb}^{k+1}_\ff=&\overline{T_k p_{k,n}}=\overline{\left\{\phi_k(\bv): v_i=b_{ii}e_i \right\}} \label{fiberthilb}\\
    \widehat{\NHilb}^{k+1}_\ff=&\overline{B_{k,n} f_{k,n}}=\overline{\left\{[\phi_1(\bv)\subset \phi_2(\bv) \subset \ldots \subset \phi_k(\bv)]: v_i=b_{ii}e_i+\ldots +b_{1i}e_1 \right\}}
\end{align}
with $b_{ii}\neq 0$ for $1\le i \le k$.

\subsection{Partitions and coordinates on the test curve model}\label{subsec:partitions} We call an integer vector 
\[\pi=(i_1,\ldots, i_r) \in \ZZ^r \text{ with } 1\le i_1 \le \ldots \le i_r \le n\]
a partition of the sum $s(\pi)=i_1+\ldots +i_r$ of length $l(\pi)=r$. 
After fixing a basis $\{e_1,\ldots, e_n\}$ of $\CC^n$, a basis of $\wedge^k (\oplus_{i=1}^k \Symm^i \CC^n)$
is given by vectors of the form
\[e_{\pi_1} \wedge \ldots \wedge e_{\pi_k}\]
encoded by sets of partitions $\{\pi_1,\ldots \pi_k\}$ with $\pi_i \neq \pi_j$ and $l(\pi_i) \le k$ for all $1\le i,j \le k$. Here for a partition $\tau=(i_1,\ldots, i_r)$ we use the notation $e_\tau=\prod_{t \in \tau}e_t \in \sym^r \CC^n$, and we associate a unit box $\boxx(\tau)$ centered at $e_{i_1}+\ldots +e_{i_r} \in \ZZ^n$. This way the $k$ boxes $\{\boxx(\pi_1),\ldots, \boxx(\pi_k)\}$ corresponding to the set $\{\pi_1,\ldots, \pi_k\}$ encode the basis vector $e_{\pi_1} \wedge \ldots \wedge e_{\pi_k}$, see Figure \ref{figure:young}. 

We call the set $\{\boxx(\pi_1),\ldots, \boxx(\pi_k)\}$ a Young diagram (or staircase) if 
\begin{equation}\label{young}
(i_1,\ldots, i_r)\in \pi,\ r>1 \Rightarrow (i_1,\ldots,\hat{i}_j, \ldots, i_r) \in \pi \text{ for all } 1\le j \le r.
\end{equation}
In \cite{bsz} we gave an alternative, combinatorial characterization of Young diagrams as follows. 
\begin{definition}
    We call the set of partitions $\pi=\{\pi_1,\ldots,\pi_k\}$ complete, if for any $l \in \{1,\ldots, k\}$ and any proper sub-partition $\tau \subsetneq \pi_l$ there is a $j \in \{1,\ldots, k\}$ such that $\pi_j=\tau$.
\end{definition}
\begin{lemma}[\cite{bsz}]
    The set of partitions $\pi=\{\pi_1,\ldots, \pi_k\}$ is complete if and only if the associated set of boxes $\boxx(\pi)=\{\boxx(\pi_1),\ldots, \boxx(\pi_k)\}$ is a Young diagram, that is, it satisfies \eqref{young}. 
\label{completeboxes}
\end{lemma}

Assume that $\{e_1,\ldots, e_n\}$ is an eigenbasis of $\CC^n$ for the action of the torus $T^n \subset \GL(n)$.  Since $\phi^\grass$ is $\GL(n)$-equivariant, torus fixed points under the $T^n \subset \GL(n)$ action on $\Hilb^{k+1}(\CC^n)$ correspond to monomial ideals 
\[\mon \in \CHilb^{k+1}_0(\CC^n) \subset \PP(\wedge^k (\oplus_{i=1}^k \Symm^i \CC^n))\]
which correspond to Young diagrams of boxes, that is complete sets of partitions $\pi=\{\pi_1,\ldots, \pi_k\}$. 

\begin{definition} For a monomial ideal $\mon \subset \CC[x_1,\ldots, x_n]$ we will denote by $\pi^\mon=\{\pi^\mon_1,\ldots, \pi^\mon_k\}$ the corresponding complete set of boxes/partitions. Conversely, the complete set $\pi=\{\pi_1,\ldots, \pi_k\}$ corresponds to the torus fixed point 
\[[e_\pi]=[e_{\pi_1} \wedge \ldots \wedge e_{\pi_k}]\in \PP(\wedge^k (\oplus_{i=1}^k \Symm^i \CC^n)\]
and we denote the corresponding monomial ideal by $\mon^\pi$. 
\end{definition}

By fixing an order of the basis elements in the eigenbasis $\{e_1,\ldots, e_n\}$, we can restrict our attention to the corresponding fiber of $\pi^\Hilb$ in Diagram \eqref{diagram:res}. By \eqref{fiberchilb} and \eqref{fiberthilb}, the fibers over $\ff$ sit in smaller subspaces, namely
\[\widehat{\CHilb}^{k+1}_\ff \subset \PP_\ff=\Span(e_{\pi_1}\wedge \ldots \wedge e_{\pi_k}: s(\pi_i) \le i \text{ for }i=1,\ldots, k)\]
and 
\[\widehat{\THilb}^{k+1}_\ff \subset \PP^T_\ff=\Span(e_{\pi_1}\wedge \ldots \wedge e_{\pi_k}: s(\pi_i) = i \text{ for }i=1,\ldots, k).\]
Following \cite{bsz}, we introduce new terminology for the combinatorial conditions appearing in $\PP_\ff$ and $\PP^T_\ff$. 
\begin{definition}\label{def:admtoric}
\begin{enumerate}
    \item We call a sequence of partitions $\pi=(\pi_1,\ldots, \pi_k)$ admissible if $\pi_i \neq \pi_j$ and $s(\pi_i)\le i$ for all $1\le i<j \le k$. We call the corresponding torus fixed point $[e_\pi]$ admissible with respect to reference flag $\bff$ defined in Definition \ref{def:widetilde}. 
    \item We call $\pi$ toric if $s(\pi_i)= i$ for all $1\le i \le k$, and we call the corresponding torus fixed point $[e_\pi]$ is admissible with respect to $\ff$. Note that toric sequences are automatically admissible. 
\end{enumerate}    
We denote by $\calP^\ff_k$ resp. $\calP^T_k$ the set of length $k$ admissible resp. toric sequences. We denote by $\PP_\ff$ (resp. $\PP^T_\ff$) the subspaces of $\PP(\wedge^k (\oplus_{i=1}^k \Symm^i \CC^n)$ spanned by admissible (resp. toric) bases. Note, in particular, that $e_{k+1},\ldots, e_n$ do not appear in admissible and toric sequences due to the $s(\pi_i)\le i$ condition.  
\end{definition}
With this terminology \eqref{fiberchilb} and \eqref{fiberthilb} can rewritten as  
\[\widehat{\CHilb}^{k+1}_\ff=\overline{\{[\Sigma_{\pi \in \calP_k^\ff} b_\pi e_\pi] : b_{ii} \in \CC^*, b_{ij} \in \CC\}}\]
and 
\begin{equation}\label{thilb}
\widehat{\THilb}^{k}_\ff=\overline{\{[\Sigma_{\pi \in \calP^T_k} b_\pi e_\pi] : b_{ii} \in \CC^*\}}.
\end{equation}
From this description we deduce the following 
\begin{lemma}\label{lemma:crucial} 
\begin{enumerate}
    \item If $\mon \in \widehat{\CHilb}^{k+1}_\ff$ then the set $\pi^\mon$ has an ordering which is admissible. 
    \item If $\mon \in \widehat{\THilb}^{k+1}_\ff$ then $\pi^\mon$ has an ordering which is toric, and this ordering is in fact unique.
\end{enumerate}
\end{lemma}

\begin{figure*}[h]
    \centering
    \begin{subfigure}[b]{0.2\textwidth}
        \centering
        \begin{ytableau}
         2   \\
 	& 1& \none & 1^3 
	\end{ytableau}
        \caption{The box diagram of \textbf{non-complete} $e_1 \wedge e_2 \wedge e_1^3 \in \wedge^3(\sym^{\le 3}\CC^2)$.}
    \end{subfigure}
    \quad
    \begin{subfigure}[b]{0.2\textwidth}  
        \centering 
        \begin{ytableau}
        2^2    \\
	2   \\
 	& 1
	\end{ytableau}
        \caption{The Young diagram of \textbf{non-admissible} $e_1 \wedge e_2 \wedge e_2^2$ corresponding to monomial ideal $(x^2,xy,y^3)\in \Hilb^4_0(\CC^2)$}
    \end{subfigure}
    \quad
    \begin{subfigure}[b]{0.2\textwidth}  
        \centering 
        \begin{ytableau}
	2   \\
 	& 1& 1^2 
	\end{ytableau}
        \caption{The Young diagram of the \textbf{admissible} $e_1 \wedge e_1^2 \wedge e_2$ corresponding to monomial ideal $(x^3,xy,y^2)\in \widehat{\CHilb}^4_\ff(\CC^2)$}
    \end{subfigure}
    \quad
    \begin{subfigure}[b]{0.2\textwidth}  
        \centering 
        \begin{ytableau}
        2  & 12 \\
 	& 1
	\end{ytableau}
        \caption{The \textbf{toric} fixed point $e_1 \wedge e_2 \wedge e_1e_2$ corresponding to monomial ideal $(x^2,y^2)\in \widehat{\THilb}^4_\ff(\CC^2)$}
    \end{subfigure}
\caption{}
\label{figure:young}
\end{figure*}
\begin{example}
The affine coordinates over the flag $\ff=(e_1 \subset \langle e_1,e_2\rangle \subset \langle e_1,e_2,e_3\rangle \subset \CC^n) \in \flag_3(\CC^n)$ are
\begin{align*}
v_1=& b_{11}e_1 \\
v_2= &b_{22}e_2+b_{12}e_1\\
v_3= &b_{33}e_3+b_{23}e_2+b_{13}e_1
\end{align*}
where $\b_{11}\b_{22}\b_{33} \neq 0$. Hence $\widehat{\CHilb}^{4}_\ff(\CC^n)$ is the closure of the locus
\begin{align*}
    \{[v_1 \wedge (v_2+v_1^2) \wedge (v_3 + 2v_1v_2+ v_1^3)]\}=& \{[b_{11}b_{22}b_{33} e_1 \wedge e_2 \wedge e_3 + b_{11}^3b_{33}e_1 \wedge e_1^2 \wedge e_3 +\\ & (b_{11}^3b_{23}-b_{11}^2 b_{12} b_{22}) e_1 \wedge e_1^2 \wedge e_2+\\ 
 &  b_{11}^2\b_{22}^2 e_1 \wedge e_2 \wedge e_1e_2+ b_{11}^4 b_{22}e_1 \wedge e_2 \wedge e_1^3) + \\
 & +b_{11}^4b_{22}e_1 \wedge e_1^2 \wedge e_1e_2 +b_{11}^6 e_1 \wedge e_1^2 \wedge e_1^3]\} \nonumber
\end{align*}
in $\grass_3(\sym^{\le 3}\CC^n)$. The toric Hilbert scheme $\widehat{\THilb}^{4}_\ff(\CC^n)$ has a simpler form with $v_i=b_{ii}e_i$, it is the closure of the locus 
\begin{multline}
    \{[b_{11}b_{22}b_{33} e_1 \wedge e_2 \wedge e_3 + b_{11}^3 b_{33}e_1 \wedge e_1^2 \wedge e_3+ b_{11}^2b_{22}^2 e_1 \wedge e_2 \wedge e_1e_2+\\
    +b_{11}^4b_{22} (e_1 \wedge e_2 \wedge e_1^3 + e_1 \wedge e_1^2 \wedge e_1e_2) + b_{11}^6 e_1 \wedge e_1^2 \wedge e_1^3]\} \nonumber
\end{multline}
\end{example}


A key geometric feature of the diagram \eqref{diagram:res} is that there are monomial ideals $\mon \in \CHilb^{k+1}_0(\CC^n)$ such that $\rho^{-1}(\mon)$ consists of various monomial ideals sitting over different flags. A simple example is $\mon=(x^2,y^2) \in \Hilb^4(\CC^2)$ and the flags 
\[\ff=(\langle e_1 \rangle \subset \langle e_1,e_2 \rangle=\CC^2) \text{ and } \ff'=(\langle e'_1 \rangle \subset \langle e'_1,e'_2 \rangle =\CC^2).\] 
where $e_1'=e_2$ and $e_2'=e_1$. Here $\rho^{-1}(\mon) \in \widehat{\CHilb}^{4}_\ff(\CC^2) \cap \widehat{\CHilb}^{4}_{\ff'}(\CC^2)$.
\begin{definition}\label{def:Tmonomial}
    We call a monomial ideal $\mon \in \Hilb^{k+1}_0(\CC^n)$ T-monomial with respect to the torus-fixed flag $\ff=(\langle e_1 \rangle \subset \ldots \subset \langle e_1,\ldots, e_k\rangle \subset \CC^n)$, if the set $\pi^\mon$ has an ordering which is a toric sequence.
\end{definition}
Figure \ref{figure:T-monomial} shows examples of $T$-monomial and not $T$-monomial ideals. 
\begin{figure*}[h]
    \centering
    \begin{subfigure}[b]{0.4\textwidth}
        \centering
        \begin{ytableau}
	2 & 12  \\
 	& 1
	\end{ytableau}
        \caption{T-monomial ideal $(x^2,y^2)$ corresponding to toric sequence ([1],[2],[12])}
    \end{subfigure}
    \quad
    \begin{subfigure}[b]{0.4\textwidth}  
        \centering 
        \begin{ytableau}
	2   \\
 	& 1& 1^2 
	\end{ytableau}
        \caption{$\mon=(x^3,xy,y^2)$ is not $T$-monomial: neither $([1],[1^2],[2])$
        nor $([1],[2],[1^2])$ is toric (although they are both admissible)}
    \end{subfigure}
\caption{}
\label{figure:T-monomial}
\end{figure*}
A crucial step in our proof of Theorem \ref{main1} is to show the converse of Lemma \ref{lemma:crucial} (2): all $T$-monomial ideals in $\Hilb^{k+1}(\CC^n)$ with respect to $\ff$ sit in $\widehat{\THilb}^{k+1}_\ff(\CC^n)$


\section{Proof of the main theorem}

The proof of Theorem \ref{main1} has three main steps. 
First, in \S \ref{subsec:step1} we prove that all $T$-monomial ideals with respect to the flag $\ff$ sit in the $\ff$-toric Hilbert scheme $\widehat{\THilb}_\ff(\CC^n)$, which is the converse of Lemma \ref{lemma:crucial} (2). We use the test curve model introduced in the previous section to carefully analyse the boundary points of the image of $\phi^\grass$. We give an alternative formulation of the proof using toric resolutions of the test curve map.

Next, in \S \ref{subsec:trivialextension} we define the socle extension of a finite dimensional algebra, which roughly corresponds to adding a new nilpotent direction to the algebra, and results an extension map  
\[\epsilon: \Hilb^k_0(\CC^n) \to \Hilb^{k+1}_0(\CC^{n+1}),\ \ A \mapsto \epsilon(A)\]
with an inverse on the image $q: \im(\epsilon) \to \Hilb^k_0(\CC^n)$ such that $(q \circ \epsilon)(A) \simeq A$.
We prove the following crucial descent result (Theorem \ref{thm:descent}) for socle extensions:
\[\epsilon(A) \in \CHilb^{k+1}_0(\CC^{n+1}) \Rightarrow A \in \CHilb^{k}_0(\CC^n).\]
The final step in \S \ref{subsec:trivialextension} is purely combinatorial: we associate to any admissible complete set of partitions $\pi^\mon=\{\pi_1,\ldots, \pi_k\}$ corresponding to the monomial ideal $\mon \in \CHilb^{k+1}_0(\CC^n)$ a toric sequence  $\pi^+=\{\pi^+_1,\ldots, \pi^+_{K}\}$ over some extended flag $\ff^+\in \flag_{1,\ldots, K}(\CC^{K})$ for some $K\ge k$, such that the corresponding monomial ideal $\mon^+=\mon^{\pi^+}\in \Hilb^{K+1}_0(\CC^K)$ satisfies that $A^+=\CC[x_1,\ldots, x_K]/\mon^+$ is an iterated socle extension of an algebra $A$ which is isomorphic to $\CC[x_1,\ldots, x_n]/\mon$. 

Then the proof of Theorem \ref{main1} in \S \ref{subsec:proof} goes as follows: by the first step we conclude that $\mon^{+}\in \widehat{\THilb}^{K+1}_{\ff^+}(\CC^{K})$ sits in the toric Hilbert scheme, and in particular, in the curvilinear Hilbert scheme, and by the descent theorem $\mon \in \CHilb^{k+1}_0(\CC^n)$.

\subsection{T-monomial ideals}\label{subsec:step1} In this section we prove 
\begin{theorem}\label{main1toric} Let $\mon \in \Hilb^{k+1}_0(\CC^n)$ be a $T$-monomial ideal with respect to the flag $\ff$ as in Definition \ref{def:Tmonomial}. Then $\mon \in \widehat{\THilb}^{k+1}_\ff(\CC^n)$.
\end{theorem}
We present two proofs using slightly different languages. The second proof was the starting point of this work, whose idea was based on the non-reductive GIT construction of \cite{bkcoh, berczithom}, although we do not use this model here.  
\begin{proof}
Let $\pi=(\pi_1,\ldots \pi_k)$ be a toric sequence of partitions with respect to $\ff$, representing the T-monomial ideal $\mon^\pi=[e_\pi]=[e_{\pi_1}\wedge \ldots \wedge e_{\pi_k}] \in \PP_\ff^T$. Recall that a sequence of partitions $\pi=(\pi_1,\dots,\pi_k)$ is toric if $s(\pi_i)=i$ for $1\le i \le k$. We will construct a (rational) $1$-parameter subgroup 
\[\lambda(t)=\begin{pmatrix}
    z^{\a_1} & & \\ & \ddots & \\ && z^{\a_k} 
\end{pmatrix}\]
with $\a_1,\ldots, \a_k \in \QQ$ such that $\lim_{z \to \infty} \phi(z^{\a_1}e_1,\ldots, z^{\a_k}e_k)=[e_\pi].$
We start with the construction for $k=2$. There are only two complete toric sequences: 
\[\pi=([1],[2]) \text{ and } \pi'=([1],[11])\]
corresponding to the points 
\[e_{\pi}=e_1 \wedge e_2 \text{ and } e_{\pi'}=e_1 \wedge e_1^2\]
Here 
\[\phi(z^{\a_1}e_1,z^{\a_2}e_2)=[z^{\a_1+\a_2}e_1 \wedge e_2 +z^{3\a_1}e_1 \wedge e_1^2] \in \PP_\ff.\]
Hence the limit point is $e_{\pi}$ if $\a_2>2\a_1$, and it is $e_{\pi'}$ if $\a_2<2\a_1$. Assume now that 
\begin{equation}\label{k=2}
\a_2>2\a_1
\end{equation}
and let $\pi=([1],[2],[\pi_3])$ be a complete toric sequence. These are 
\[([1],[2],[12]) \text{ and } ([1],[2],[3]).\]
The coefficients of these in $\phi(z^{\a_1}e_1,z^{\a_2}e_2,z^{\a_3}e_3)$ are $z^{\a_1+\a_2}\cdot z^{\a_1+\a_2}$ and $z^{\a_1+\a_2}\cdot z^{\a_3}$ respectively. Hence for any solution of the system of linear inequalities 
\begin{eqnarray*}
\a_2>2\a_1 \\
\a_1+\a_2 > \a_3
\end{eqnarray*}
the limit is $([1],[2],[12])$, and for a solution of  
\begin{eqnarray*}
\a_2>2\a_1 \\
\a_1+\a_2 < \a_3
\end{eqnarray*}
the limit point is $([1],[2],[12])$. We see that due to the independent variable $\a_3$, both systems have solutions. 
We follow this argument in general, and prove that the completeness of the sequence $\pi$ guarantees that there is a solution for the corresponding system of linear inequalities which gives $e_\pi$ as limit point. Assume we proved this for all complete sequences of length at most $k-1$, and let $\pi=(\pi_1,\ldots, \pi_k)$ be a complete toric sequence. The coefficient of $e_{\pi}=e_{\pi_1} \wedge \ldots \wedge e_{\pi_k}$ in $\phi(z^{\a_1}e_1, \ldots ,z^{\a_k}e_k)$ is $z^{\sum_{l=1}^k \a_{\pi_l}}$ where $\a_{\pi_l}=\sum_{i\in \pi_l}\a_i$ is a linear form. Hence the system of linear inequalities whose solutions provide the limit $e_{\pi}$ is 
\[\a_{\pi_l} > \a_{\pi'_l} \text{ for all } 1\le l \le k \text{ and } \pi'_l \in \Pi_l \tag{(1)-(k)}\]
where 
\[\Pi_l=\{(i_1,\ldots, i_r)\in \ZZ^r:1\le i_1\le \ldots \le i_r \le k, i_1+\ldots+i_r=l\}\]
By assumption we know that the first $k-1$ inequalities have a nonempty solution set. Assume that (1)-(k) does not have solution. This means that there is a $\pi'_k \in \Pi_k$ such that 
\[\a_{\pi_l} > \a_{\pi'_l} \text{ for all } 1\le l \le k-1 \text{ and } \pi'_l \in \Pi_l \tag{(1)-(k-1)}\]
implies that 
\[\a_{\pi_k}=\sum_{j \in \pi_k}\a_j \le \sum_{j' \in \pi'_k}\a_{j'}=\a_{\pi'_k}. \tag{(k)}\]
This can happen in two ways. First, if $\pi_k=[k]$ then inequality (k) contains the new variable $\a_k$ on the left hand side but neither in $\pi_k'$, nor in (1)-(k-1), so with appropriate choice $\a_k$ we get a solution. 

Otherwise, there must be a decomposition 
\[\pi_k=\tau_1 \cup \tau_2 \text{ and } \pi'_k=\tau'_1 \cup \tau'_2\]
into nontrivial sub-partitions such that w.l.o.g. $\tau_1 \neq \tau'_1$ and  
\[1< s(\tau_1)=s(\tau'_1) <k \text{ and hence } 1<s(\tau_2)=s(\tau'_2) <k.\]
and
\begin{equation*}
\sum_{j \in \tau_1}\a_j \le \sum_{j' \in \tau'_1}\a_{j'} \text{ and }
    \sum_{j \in \tau_2}\a_j \le \sum_{j' \in \tau'_2}\a_{j'}
\end{equation*}  
However, the first inequality contradicts to inequality (1)-(k-1) with $l=s(\tau_1)$ and $\pi'=\tau_1'$: indeed, by completeness of $(\pi_1,\ldots \pi_{k-1})$ we have $\pi_l=\tau_1$ and (1)-(k-1) gives $\a_{\tau_1}>\a_{\tau_1'}$.
\end{proof}

This proof can be reformulated using resolutions of the test curve map $\phi_k$. We present this reformulation here, which was originally motivated by \cite{berczithom}. Recall the toric Hilbert scheme from \eqref{fiberthilb}
\[\widehat{\THilb}^{k+1}_\ff=\overline{T_k p_{n,k}}=\overline{\left\{\phi_k(\bv): v_i=b_{ii}e_i \right\}} \subset \PP^T_\ff=\Span(e_{\pi}: \pi \text{ is toric}).\]
Here $\phi_k$ is the test curve model map $(b_1,\dots,b_k) \xmapsto{\phi_{k}} [\Sigma_{\pi \text{ toric }} b_\pi e_\pi]$
where we write $b_i=b_{ii}$ for simplicity, and for $\pi=(\pi_1,\ldots, \pi_k)$ we set $b_\pi = b_{\pi_1} \cdots b_{\pi_k}$, which is a monomial in the $b_i's$, see \eqref{thilb}. 
We will follow terminology of toric blow-ups, see e.g \cite{ToroidalEmbeddings}. 
We write $X_k := \Spec(S_k)$ for the affine part of the source space of $\phi_k$, where $S_k = \CC[b_i : 1 \leq i \leq k]$, of dimension $\Dim \, X_k = k$. In the remaining part of this section every sequence of partitions $\pi = (\pi_1,\dots,\pi_k)$ will be assumed toric, and we denote the set of these by $\calp^T_{k}$, see Definition \ref{def:admtoric}.

For a partition $\tau=(i_1,\ldots, i_r)=(1^{j_1},\ldots, k^{j_k})$ we denote by $v_\tau=(j_1,\ldots j_k)\in \ZZ^k$ the vector of exponents in the integer lattice with dual space $(\ZZ^k)^* \subset (\RR^k)^*$. Via the canonical pairing
$$ \ZZ^k \times (\ZZ^k)^* \to \ZZ , \quad (v,\mu) \mapsto \langle v, \mu \rangle $$
we introduce hyperplanes 
$$ H_{\pi,\pi'} = \{\mu \in (\RR^k)^* : \langle v_{\pi'} - v_{\pi} , \mu \rangle = 0\} \subset (\RR^k)^*$$
and write $H_{\pi,\pi'}^-=\{\mu: \langle v_\pi - v_{\pi'} , \mu \rangle \geq 0\}$. The hyperplanes $H_{\pi,\pi'}$ go through the origin and yield a division of $(\RR^k_{\geq 0})^*$ into strictly convex rational polyhedral cones (meaning that each cone contains no positive dimensional linear subspace, and is bounded by \emph{finitely many} hyperplanes); that is, we obtain a fan $F$ with support $(\RR^k_{\geq 0})^*$ and associated toric variety $\calT(F)$. Every such fan $F$ has a regular (meaning that all rays are in the integral lattice $(\ZZ^k)^*$) refinement $\tilde F$,  see Chapter 1, Theorem 11 of \cite{ToroidalEmbeddings}. Equivalently the associated toric variety $\calT(\tilde F)$ is nonsingular. We obtain the following resolution picture

$$
\xymatrix{
\calT(\tilde F) \ar[d] \ar[rdd]^-{\tilde\phi_{k}}  & \\
\calT(F) \ar[d] \ar[rd]^-{\bar\phi_{k}}   & \\
X_k \ar@{-->}[r]^-{\phi_k} & \widehat{\THilb}^{k+1}_\ff}
$$ 
where all maps are torus equivariant with respect to the torus $T^k$ acting on $S_k$ diagonally. 

\begin{proposition}\label{toricInductionStep} Let $\pi=(\pi_1,\dots,\pi_k)$ be a toric sequence of partitions, and suppose that 
\begin{enumerate}
    \item $\bigcap_{(\pi'_1,\dots,\pi'_{k-1})\in \calp^T_{k-1}} H^-_{(\pi_1,\dots,\pi_{k-1}),(\pi'_1,\dots,\pi'_{k-1})} \cap (\RR_{\geq  0}^k)^* \neq \emptyset$
    \item $\bigcap_{\pi' \in \calp^T_k} H^-_{\pi,\pi'} \cap (\RR_{\geq  0}^k)^* = \emptyset.$ 
\end{enumerate}
Then $\pi$ is not complete.
\end{proposition}

\begin{proof} Assume for contradiction that the two conditions of the proposition hold, and $\pi$   
is complete.
By (2) there must exist a toric sequence $\pi' = (\pi_1,\dots,\pi_{k-1},\pi'_k)$ with $\pi_k' \neq \pi_k$ such that
$$ \mu \in \bigcap_{(\pi'_1,\dots,\pi'_{k-1})} H^-_{(\pi_1,\dots,\pi_{k-1}),(\pi'_1,\dots,\pi'_{k-1})} \cap (\RR_{\geq  0}^k)^* \implies \langle \mu, v_{\pi_k}-v_{\pi'_k} \rangle > 0.$$
It follows that in $\pi_k$ and $\pi'_k$ we have subpartitions $\d$ in $\pi_k$ and $\d'$ in $\pi'_k$ with $s(\d) = s(\d')<k$ and satisfying
$$ \mu \in \bigcap_{(\pi'_1,\dots,\pi'_{k-1})} H^-_{(\pi_1,\dots,\pi_{k-1}),(\pi'_1,\dots,\pi'_{k-1})} \cap (\RR_{\geq  0}^k)^* \implies \langle \mu, v_{\d'}-v_{\d} \rangle < 0.$$
and we conclude that $\d \neq \pi_1,\dots,\pi_{k-1}$.

\end{proof}





Theorem \ref{main1toric} follows now from Proposition \ref{toricInductionStep}. For a complete and toric sequence $\pi$, we must show that there is a nonempty cone $\s_\pi$ in the fan $F$ satisfying 
$$\s_\pi \subset \bigcap_{\pi' \in \calp^T_k}H_{\pi,\pi'}^-.$$
The argument goes by induction on $k$. The statement is trivial for $k=1$ since $\pi = (1)$ is the only toric sequence of one partition. The induction hypothesis is now that
$$\bigcap_{(\pi'_1,\dots,\pi'_{k-1})\in \calp^T_{k-1}} H^-_{(\pi_1,\dots,\pi_{k-1}),(\pi'_1,\dots,\pi'_{k-1})} \cap (\RR_{\geq 0}^k)^* \neq \emptyset.$$
Since $\pi$ is complete, it follows from Proposition \ref{toricInductionStep} that also 
$$\bigcap_{\pi' \in \calp^T_k}H_{\pi,\pi'}^- \cap (\RR_{\geq 0}^k)^* \neq \emptyset,$$
yielding in particular the existence of such cone $\s_\pi$ in $F$ as we wanted.


\subsection{Socle extensions and the descent theorem}\label{subsec:trivialextension}




Let $\Alg_k$ denote the set of unital, commutative, associative algebras over $\CC$ of dimension $k$. Recall that the socle of $\xi \in \Alg_k$ is defined as
$$\Soc(\xi) = \{x \in \xi: a\cdot x = 0 \text{ for all } a \in \xi\}.$$ 
which is the same as the annihilator sub-algebra of $\xi$ (not the annihilator as a module).

Choosing generators $E_1,\dots,E_k$ for $\xi \in \Alg_k$, the algebra structure $\xi$ is determined by its structure constants $c_{ij}^l$ in the equations
$$E_i \cdot_\xi E_j = \sum_l c_{ij}^k E_l.$$
Then $\xi = \CC[E_1,\dots,E_k]/I_\xi$ for the ideal of relations $I_\xi$. 
\begin{definition}\label{def:trivialextension}  We say that $\xi'\in \Alg_{k+m}$ is an ($m$-times) socle extension of $\xi \in \Alg_k$ if $\xi'=\langle E_1,\dots,E_k,F_1,\dots,F_m \rangle$ is generated by $k+m$ generators with relations  
\begin{equation}\label{socleext}
 E_i \cdot_{\xi'} E_j = E_i \cdot_{\xi} E_j \in \xi=\langle E_1,\dots,E_k \rangle, \quad E_i\cdot_{\xi'}F_j = 0 , \quad F_i\cdot_{\xi'}F_j = 0. 
 \end{equation}
\end{definition}
In other words \eqref{socleext} can be rewritten as
\begin{enumerate}
    \item The subalgebra generated by $E_1,\ldots ,E_k$ is isomorphic to $\xi$ and
    \item $F_1,\dots,F_m \in \Soc(\xi')$.
\end{enumerate}    
If $\xi \in \Alg_k(\CC^n)$ is presented as $\xi =\CC[x_1,\ldots, x_n]/I_\xi \in \Hilb^k(\CC^n)$, then there is a canonical presentation for its $m$-times socle extension $\xi' \in \Alg_{k+m}$ in $\Hilb^{k+m}(\CC^{n+m})$ by choosing the $m$ extension directions as new coordinate directions.
\begin{definition}\label{def:socle} Let $\CC^n=\Span_\CC(e_1,\ldots, e_n) \subset \CC^{n+m}=\Span_\CC(e_1,\ldots, e_n,f_1,\ldots, f_m)$. We say that $\xi'=\CC[E_1,\ldots, E_k,F_1,\ldots, F_m]/I_{\xi'} \in \Hilb^{k+m}(\CC^{k+m})$ is the socle extension of $\xi=\CC[E_1,\ldots, E_k]/I_\xi \in \Hilb^k(\CC^n)$ if
$E_1,\ldots, E_k \in \CC^n$ and $F_i=f_i$ for $1\le i \le m$.  
\end{definition}

 
\begin{figure}[ht!]
\centering
\includegraphics[width=50mm]{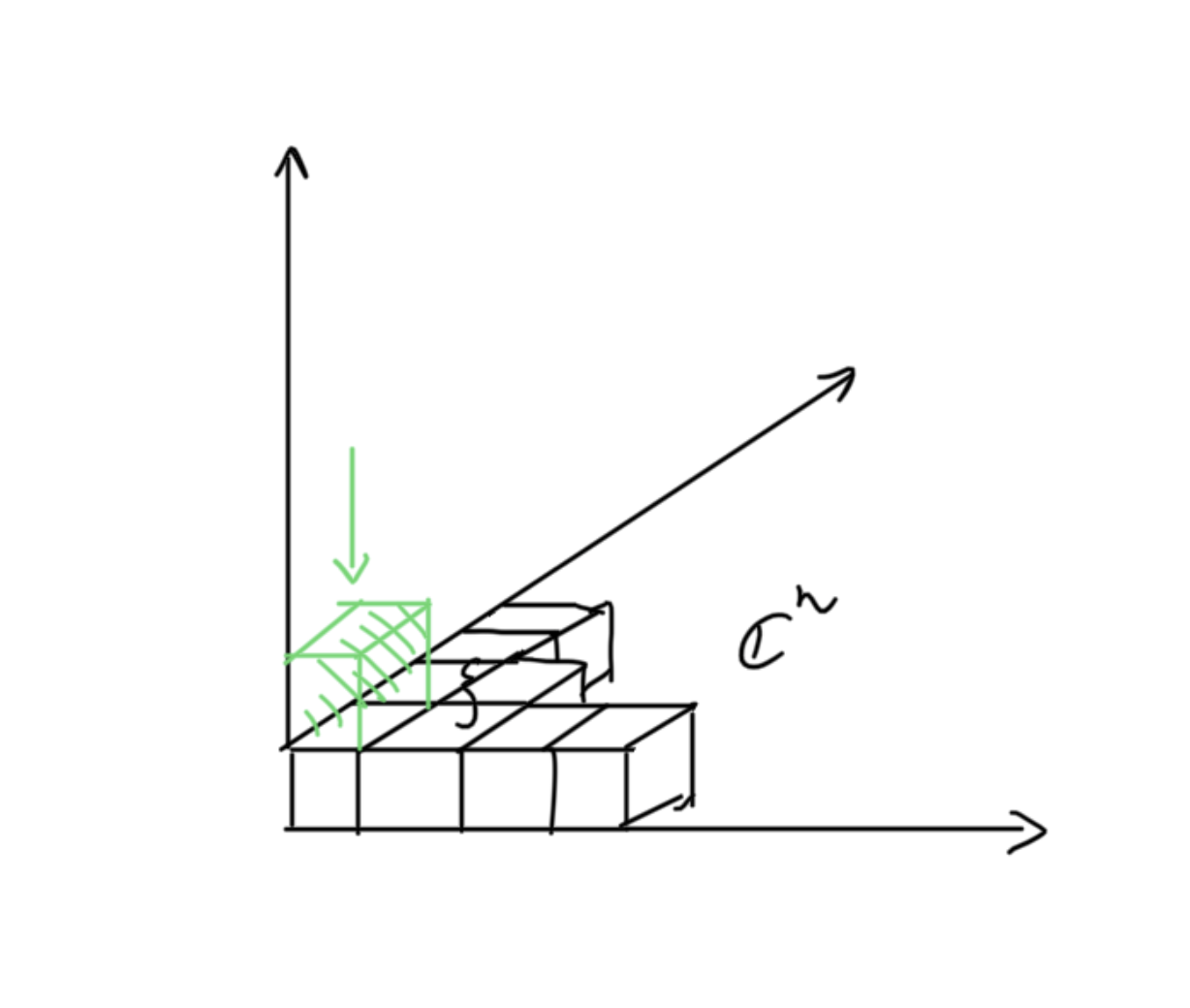}
\caption{The socle extension $\xi'=[\xi \wedge e_3]\in \Hilb^{9}_0(\CC^3)$ of $\xi \in \Hilb^{8}_0(\CC^2)$ in the new vertical $e_3$ coordinate direction} 
\label{figure:socle}
\end{figure}
According to Remark \ref{remark:embed}, $\Hilb^{k+1}_0(\CC^n)\subset \grass_{k}(\symdot)\subset \PP(\wedge^{k}\symdot)$, and the socle extension of $\xi \in \PP(\wedge^{k}\symdot)$ is $[\xi \wedge f_1 \wedge \ldots \wedge f_m] \in \PP(\wedge^{k+m}(\sym^{\le k+m}\CC^{n+m}))$.
Figure \ref{figure:socle} illustrates the $m=1$ situation: geometrically, $\xi'$ is the algebra arising when a $k+1$th point collides into $\xi$ from a new coordinate direction. It is clear that a socle extension of a curvilinear algebra is curvilinear. The following key theorem asserts that the converse is valid as well.  
\begin{theorem}[Descent theorem]\label{thm:descent}
	Let $k \le n$ and let $\xi' \in \Hilb^{k+m}_0(\CC^{n+m})$ be a socle extension of $\xi \in \Hilb^k_0(\CC^n)$. If $\xi' \in \CHilb^{k+m}_0(\CC^{n+m})$ then $\xi \in \CHilb^k_0(\CC^n)$.
\label{curvTrivialExtension}
\end{theorem}

\begin{proof}
We prove the theorem for $m=1$ and apply induction. Let $\xi' \in \Hilb^{k+1}_0(\CC^{n+1})$ be a socle extension of $\xi \in \Hilb^k_0(\CC^n)$. We pick basis elements $e_1,\ldots, e_{n},f$ on $\CC^{n+1}$ according to Definition \ref{def:socle}, that is 
\begin{equation}\label{eq:xi} 
	\xi \in \Span(e_1,\ldots, e_{n}),  \text{ and } \xi'/\xi= \Span(f)
\end{equation}
Since $\xi' \in \CHilb^{k+1}_0(\CC^{n+1})$, Theorem \ref{bszmodel} says that 
\[\xi' \in \overline{\mathrm{im}(\phi^\grass)}=\overline{\left\{[v_1 \wedge (v_2+v_1^2) \wedge \ldots \wedge (v_k+\Sigma_{\bi \in \Pi_k}v_\bi)]: v_i \in \CC^{n+1} \right\}} \subset \PP(\wedge^k \sym^{\le k}(\CC^{n+1}))\]
and \eqref{eq:xi} means that $\xi'$ has the form 
\begin{equation*} 
\xi'=[\xi \wedge f] \text{ where } \xi \in \wedge^{k-1}\sym^{\le k-1} \CC^n \text{ and } \CC^n=\mathrm{Span}(e_1,\ldots, e_n).
\end{equation*}
More precisely, there is a $k-1$-dimensional subspace $V_{k-1}\subset \CC^n$ such that 
\begin{equation}\label{xiwedgeek} 
\xi'=[\xi \wedge f] \text{ where } \xi \in \wedge^{k-1}\sym^{\le k-1} V_{k-1} 
\end{equation}
Hence there is a $2\le l \le k$ such that $\xi'$ sits in the boundary subset $\mathcal{B}_l$ defined as 
 \[\mathcal{B}_l=\overline{\left\{[v_1 \wedge (v_2+v_1^2) \wedge \ldots \wedge (v_{l-1}+\Sigma_{\bi \in \Pi_{l-1}}v_\bi) \wedge f \wedge (v_{l+1}+\Sigma_{\bi \in \Pi_{l+1}}v_\bi) \wedge \ldots \wedge (v_k+\Sigma_{\bi \in \Pi_k}v_\bi)]\right\}}\]
 where $v_i \in \CC^n=\langle e_1,\ldots, e_{n} \rangle$ for $i \neq l$. 
 
Next we show that $\mathcal{B}_l \subset \mathcal{B}_k$ for all $2\le l \le k$.  We recall from Definition \ref{def:widetilde} the partial blow-up of the curvilinear Hilbert scheme, which fibers over $\flag_{1,\ldots, k}(\CC^{n+1})$: we let $B_{k,n+1} \subset \GL(n+1)$ denote the parabolic subgroup which preserves the reference flag
\[\mathbf{f}=(\langle e_1 \rangle \subset \langle e_1,e_2 \rangle \subset  \ldots  \subset \langle e_1,\ldots e_{k-1} \rangle \subset \langle e_1,\ldots e_{k-1},f \rangle \subset  \CC^{n+1}) \in \flag_{1,\ldots, k}(\CC^{n+1})\] 
and  
\[\widehat{\CHilb}^{k+1}_0(\CC^{n+1})=\GL(n+1) \times_{B_{k,n+1}} \overline{B_{k,n+1}\cdot p_{k,n+1}}.\] 
is the fiberwise compactification in $\grass_k(\sym^{\le k}\CC^{n+1})$, which is a partial resolution of $\CHilb^{k+1}(\CC^{n+1})$. We study the following extension of Diagram \eqref{diagram:res}:
\[\xymatrix{\widehat{\CHilb}^{k+1}_0(\CC^{n+1}) \ar@{->>}[r]^-\rho \ar[d]^{\pi^\Hilb} \ar@/^3pc/[dd]^\pi   &  \CHilb^{k+1}_0(\CC^{n+1})=\overline{\GL(n+1)\cdot p_{k,n+1}}\\
\flag_{1,\ldots, k}(\CC^{n+1}) \ar[d]^{\pi^\flag} & \\ \flag_{k-1, k}(\CC^{n+1}) & }\]
Here the projection 
\[\pi^\flag: \flag_{1,\ldots, k}(\CC^{n+1}) \to \flag_{k-1, k}(\CC^{n+1}), \ \  (V_1 \subset \ldots \subset V_k \subset \CC^{n+1}) \mapsto (V_{k-1} \subset V_k \subset \CC^{n+1})\]
forgets the first $k-2$ subspaces in the flag and $\pi=\pi^\flag \circ \pi^\Hilb$ holds. 

Next, we assign to the flags $\bV=(V_1 \subset \ldots \subset V_k \subset \CC^{n+1}) \in \flag_{1,\ldots, k}(\CC^n)$ and $\bar{\bV}=(V_{k-1} \subset V_k \subset \CC^{n+1})\in \flag_{k-1,k}(\CC^{n+1})$ the subsets 
\begin{multline*}
\sym^\bV(\CC^{n+1})=V_1 \wedge (V_2 \oplus \sym^2(V_1)) \wedge \ldots \wedge (V_k \oplus \sym^2(V_{k-1}) \oplus \ldots \oplus \sym^k(V_1)))= \\
=\wedge_{i=1}^k \oplus_{j=1}^i \sym^j(V_{i+1-j}) \subset \wedge^k(\sym^{\le k}\CC^{n+1})
\end{multline*}
and 
\begin{multline*} 
\sym^{\bar{\bV}}(\CC^{n+1})=V_{k-1} \wedge (V_{k-1} \oplus \sym^2(V_{k-1})) \wedge \ldots \wedge (V_{k} \oplus \sym^2(V_{k-1}) \oplus \ldots \oplus \sym^k(V_{k-1})))= \\
\wedge_{i=1}^{k-1} \sym^{\le i}V_{k-1} \wedge (V_k+\oplus_{j=2}^k\sym^{\le j}V_{k-1}) \subset \wedge^k(\sym^{\le k}\CC^{n+1})
\end{multline*}
Note that $\sym^\bV(\CC^{n+1}) \subset \sym^{\bar{\bV}}(\CC^{n+1})$. For a point 
\[\zeta=[\zeta_1 \wedge \zeta_2 \wedge \ldots \wedge \zeta_k]\in \CHilb^{k+1}_0(\CC^{n+1}) \subset \PP(\wedge^k(\sym^{\le k}\CC^{n+1}))\]
its preimage $\rho^{-1}(\zeta) \subset \widehat{\CHilb}^{k+1}_0(\CC^{n+1})$ might intersect several fibers of $\pi^\Hilb$ and $\pi$. By Theorem \ref{bszmodel} these can be characterised as follows: 
\begin{equation}\label{characterisation1}
\rho^{-1}(\zeta)\cap (\pi^\Hilb)^{-1}(\bV) \neq \emptyset \text{ if and only if } \zeta \in \PP(\sym^\bV(\CC^{n+1}))
\end{equation}
and 
\begin{equation}\label{characterisation2}
\rho^{-1}(\zeta) \cap \pi^{-1}(\bar{\bV}) \neq \emptyset  \text{ if and only if } \zeta \in \PP(\sym^{\bar{\bV}}(\CC^{n+1}))
\end{equation}
Let $\mathbf{f}$ stand for the reference flag 
\[\bar{\mathbf{f}}:=\pi^\flag(\mathbf{f})=(\langle e_1,\ldots, e_{k-1} \rangle \subset   \langle e_1,\ldots e_{k-1},f \rangle \subset  \CC^n) \in \flag_{k-1, k}(\CC^{n+1}).\]
In \eqref{xiwedgeek} we can assume w.l.o.g that $V_{k-1}=\mathrm{Span}(e_1,\ldots, e_{k-1})$ and hence 
\[\xi' \in \PP(\sym^{\bar{\mathbf{f}}}(\CC^{n+1})).\]
Then by \eqref{characterisation2} $\rho^{-1}(\xi') \cap \pi^{-1}(\bar{\mathbf{f}})$ is nonempty. Let $\hat{\xi'}$ be a point in this intersection. This means that 
\[\hat{\xi'}=[v_1 \wedge (v_2+v_1^2) \wedge \ldots \wedge (v_k+\Sigma_{\bi \in \Pi_k}v_\bi)]\]
with some $v_1,\ldots, v_{k-1} \in \langle e_1,\ldots, e_{k-1} \rangle,  v_k \in \langle e_1,\ldots, e_{k-1},f \rangle$.
But then, since $\pi^\flag$ is surjective, there is a full flag   
\begin{multline}
\mathbf{f}_{\xi'}=(\langle w_1 \rangle  \subset \langle w_1,w_2 \rangle \subset  \ldots  \subset \langle w_1,\ldots w_{k-1} \rangle = 
\langle e_1,\ldots, e_{k-1} \rangle \subset \\ \subset \langle w_1,\ldots, w_{k-1},f \rangle \subset \CC^{n+1}) \in (\pi^\flag)^{-1}(\bar{\mathbf{f}})
\end{multline}
defined by $w_1,\ldots, w_{k-1} \in \mathrm{Span}(e_1,\ldots, e_{k-1})$. Let $B_{\bw,f} \subset \GL(n+1)$ denote the stabiliser of $\mathbf{f}_{\xi'}$. Then 
\begin{equation}\label{xiprime}
\hat{\xi'} \in (\pi^\Hilb)^{-1}(\mathbf{f}_{\xi'})=\overline{B_{\bw,f} \cdot [w_1 \wedge (w_2+w_1^2) \wedge \ldots \wedge (w_{k-1}+\Sigma_{\bi \in \Pi_{k-1}}w_\bi) \wedge (f+\Sigma_{\bi \in \Pi_k} w_{\bi}]}.
\end{equation}
Hence 
\[\xi'=[\xi \wedge f]=\rho(\hat{\xi'})\in \overline{B_{\bw,f} \cdot [w_1 \wedge (w_2+w_1^2) \wedge \ldots \wedge (w_{k-1}+\Sigma_{\bi \in \Pi_{k-1}}w_\bi) \wedge (f+\Sigma_{\bi \in \Pi_k} w_{\bi}]},\]
and since $\Span(w_1,\ldots, w_{k-1})=\mathrm{Span}(e_1,\ldots, e_{k-1})$, this can only happen if 
\begin{equation}\label{eq:final}
\xi' \in \overline{B_{\bw,f} \cdot [w_1 \wedge (w_2+w_1^2) \wedge \ldots \wedge (w_{k-1}+\Sigma_{\bi \in \Pi_{k-1}}w_\bi) \wedge f]}.
\end{equation}
Let 
\[\bff_\xi=(\langle w_1 \rangle  \subset \langle w_1,w_2 \rangle \subset  \ldots  \subset \langle w_1,\ldots w_{k-1} \rangle \subset \CC^n)\] 
be the truncated flag after dropping $f$, and let $B_\bw \subset \GL(n)$ denote its stabiliser. By \eqref{xiwedgeek} and \eqref{eq:final}
\[\xi \in \overline{B_{\bw} \cdot [w_1 \wedge (w_2+w_1^2) \wedge \ldots \wedge (w_{k-1}+\Sigma_{\bi \in \Pi_{k-1}}w_\bi)]} \subset \CHilb^k_0(\CC^n).\]
\end{proof}

	
\begin{remark}
	In the above proof it is neccessary that $\xi$ is a subalgebra of $\xi'$. In particular, it is not enough that $F$ is a socle element of $\xi'$ as the following example shows. By \cite{Cartwright2009} there exists an algebra $A \notin \CHilb^8(\AAA^4)$ with Hilbert function $H_A = (1,4,3)$ (in fact, $A$ can be chosen not to be even smoothable).
Pick $d$ maximal so that there exists an algebra $A$ with Hilbert function $H_A = (1,4,d)$ such that $A \notin \CHilb^{d+5}(\AAA^4)$. If $d=10$ (which is the maximal possible choice) we have $A \simeq C[x_1,x_2,x_3,x_4]/m^3 \in \CHilb^{15}(\AAA^4)$ for $m=(x_1,x_2,x_3,x_4)$ the maximal ideal, hence it follows that $3 \leq d < 10$.

Let now $A \notin \CHilb^{d+5}(\AAA^4)$ be such algebra with $H_A = (1,4,d)$ for this maximal $d$. Then $A$ may be written as $A = B / (s)$ for some some $B$ with $H_B=(1,4,d+1)$ and some socle element $s \in \Soc(B)$ so that by the choice of $d$ we have $B \in \CHilb^{d+6}(\AAA^4)$.
\end{remark}

\subsection{T-extensions of monomial ideals}\label{toricExtensionSection}

Let $\mon \in \Hilb^{k+1}_0(\CC^n)$ be a monomial ideal. In this section we construct a monomial ideal $\mon^+ \in \Hilb^{K+1}_0(\CC^K)$ with some $K\ge k$, such that (i) $\mon^+$ is $T$-monomial and (ii) $A^+=\CC[x_1,\ldots, x_K]/\mon^+$ is a socle extension of an algebra $A$ which is isomorphic to $\CC[x_1,\ldots, x_n]/\mon$.

Let $\pi^{\mon}=\{\pi_1^\mon, \ldots, \pi_k^\mon\}$ be the complete set of partitions and $Y^\mon$ the Young diagram formed by the boxes $\{\boxx(\pi_1),\ldots, \boxx(\pi_k)\}$ as in \S \ref{subsec:partitions}. We define the reverse lexicographic order on the lattice $\ZZ^n$ as follows:
\[(i_1,\ldots, i_n) <_{\RL} (i_1',\ldots, i_n') \iff  i_d < i_d' \text{ for } d:=\text{max} \{j : i_j \neq i_j'\}.\]
\begin{definition} Let $\pi^\mon_\RL=(\pi^\mon_1 <_\RL \ldots <_\RL \pi^\mon_k)$ denote the sequence of partitions in $\pi^\mon$ in increasing order w.r.t. the reverse lexicographic order. We also call this the $\RL$-order of boxes of $Y^\mon$, or an $\RL$-sequence. 
\end{definition}

\begin{remark}
\begin{enumerate}
   \item   The $\RL$-sequence $\pi^\mon_\RL$ defined above is complete, but not necessarily admissible over $\ff$. Indeed in the example in Figure \ref{figure:notadm} the $\RL$-order of boxes is 
   \[\pi_{\RL}^\mon=([1],[1^2],[2],[2^2],[2^3])\]
   which is not admissible as $s(2^3)=6>5$.
\begin{figure}
\begin{center}
\begin{ytableau}
2^3    \\
2^2    \\
2   \\
 & 1& 1^2 
\end{ytableau}
\end{center}
\caption{Non-admissible $\RL$-order: $([1],[11],[2],[22],[222])$}
\label{figure:notadm}
\end{figure}

\item The RL-sequence might not be admissible, but there are several ways to construct an admissible order of the boxes in $Y^\mon$. Such admissible orders can be constructed from the nilpotent filtrations appearing in Kazarian's construction of the non-associative Hilbert scheme \cite{kazarian3}. 
\end{enumerate}
\end{remark}

\begin{figure}
\begin{center}
\begin{ytableau}
*(red) 2^4     \\
*(red) 2^3 & *(red) 12^3   \\
*(red) 2^2 & *(red) 12^2     \\
*(red) 2& *(red) 12& *(red) 1^22  \\
*(red)  & *(red) 1& *(red) 1^2& *(red) 1^3 

\end{ytableau} \hspace{3cm} 
\begin{ytableau}
*(red) 4^4 & *(blue) 17 & *(blue) 18 & *(blue) 19 \\
*(red) 4^3 & *(red) 14^3 & *(blue) 14 & *(blue) 15  \\
*(red) 4^2 & *(red) 14^2& *(blue) 10 & *(blue) 11  \\
*(red) 4 & *(red) 14& *(red) 1^24& *(blue) 7  \\
*(red)  & *(red) 1& *(red) 1^2& *(red) 1^3
\end{ytableau}
\end{center}
\caption{$\RL$ sequence and the corresponding toric extension}
\label{figure:example}
\end{figure}

Next, we assign to $\pi^\mon_\RL$ a toric sequence $\pi^\mon_+$ over an extended flag $\bff^+$ defined along the construction. 
We explain this purely combinatorial construction through an example first. Let 
\[\mon=(x_1^4,x_1^3x_2,x_1^2x_2^2,x_1x_2^4,x_2^5) \in \Hilb^{12}_0(\CC^2)\]
be the red monomial ideal on the left hand side of Figure $\ref{figure:example}$, corresponding to the complete set 
\[\pi^\mon=\left\{[1],[2],[1^2],[12],[2^2],[1^3],[1^22],[12^2],[2^3],[12^3],[2^4]\right\}.\] 
The $\RL$-order of partitions in $\pi^\mon$ is the following: we list the bottom row from left to right, followed by the second row from left to right, etc. This gives
\[\pi^\mon_\RL=([1],[1^2,[1^3],[2],[12],[1^22],[2^2],[12^2],[2^3],[12^3],[2^4]).\]
We construct the associated toric sequence by inserting into $\pi^\mon_\RL$ the blue boxes, which complement the original red partition to the embedding $r_1 \times c_1$ rectangle where $r_1$ is the length of the bottom row and $c_1$ is the length of the first column in $Y^\mon$. The rule is the following. We index the boxes in the embedding rectangle from $0$ to $r_1c_1-1=24$ following the reverse lexicographic order (i.e starting from bottom to up and left to right). The first box in the second row (from below) now gets the partition equal to its index, i.e $\pi^+_{r_1+1}=[r_1+1]=[4]$ instead of $[2]$, and we use the substitution $2 \to r_1+1=4$ in the rest of the red $Y^\mon$. The partitions in the blue boxes are simply their indices. Hence the toric extension in this example is 
\begin{multline*}
   \pi^+=([1],[1^2],[1^3],[4],[14],[1^24],[7],[4^2],[14^2],[10],[11],\\
   [4^3],[14^3],[14],[15],[4^4],[17],[18],[19]) 
\end{multline*}
This is indeed a toric sequence over the flag $\bff^+=(e_1,e_2,\ldots, e_{19})$, because $s(\bi_j)=j$ for $1\le j \le 19$ and it corresponds to the  monomial ideal 
\begin{align*}
\mon^+ := \mon^{\pi^+}=\big(
&x_1^4,x_1^3x_4,x_1^2x_4^2,x_1x_4^4,x_4^5,x_2,x_3,x_5,x_6,x_8,x_9,x_{12},x_{13},x_{16}, \\
&(x_7,x_{10},x_{11},x_{14},x_{15},x_{17},x_{18},x_{19})(x_1,x_4,x_7,x_{10},x_{11},x_{14},x_{15},x_{17},x_{18},x_{19}) 
\big)
\end{align*}
sitting in $\Hilb^{20}_0(\CC^{19})$ (the last generator is the product of two ideals, i.e generated by pairwise product of generators of the factors). The red boxes give the complete set 
\[\pi^{\underline{\mon}}=\left\{[1],[1^2],[1^3],[14],[1^24],[4^2],[14^2],[4^3],[14^3],[4^4]\right\})\]
corresponding to the monomial ideal $\underline{\mon}\in \Hilb^{12}_0(\Span_\CC(e_1,e_4))=\Hilb^{12}_0(\CC^2)\subset \Hilb^{12}_0(\CC^{19})$. Then 
\begin{enumerate}
\item $\CC[x_1,\ldots x_{19}]/\underline{\mon}\simeq \CC[x_1,\ldots, x_{19}]/\mon$
\item $A^+=\CC[x_1,\ldots, x_{19}]/\mon^+$ is the iterated socle extension of $A=\CC[x_1,\ldots, x_{19}]/\underline{\mon}$ in the directions spanned by the blue boxes.
\end{enumerate}

The general construction follows the same lines. Let $\pi^\mon_\RL=(\pi^\mon_1,\ldots, \pi^\mon_k)$ be the RL-sequence corresponding to $\mon\in \Hilb^k(\CC^n)$. Let $\mathfrak{m}/(\mon+\mathfrak{m}^2)=\Span(e_1,\ldots, e_d)$, that is, $\mon$ corresponds to an $d$-dimensional Young diagram in the direction $e_1,\ldots, e_d$. Then the partitions $[1],\ldots, [d]$ are elements of the complete set $\pi^\mon$. Let $1\le \beta_i \le k$ be the position index of $[i]$ in the sequence $\pi^\mon_\RL$, that is
\[\pi_{\beta_i}=[i] \text{ for } 1\le i \le d.\] 
For $1\le j \le k$ let $\tau_j$ be the partition which we obtain by modifying the $j$th partition in the $\RL$ sequence $\pi^\mon_{\RL}$ as follows: we replace all occurrences of $[i]$ with $[\beta_i]$ for all $1\le i \le d$.

Next, as in the example above, we form the smallest $d$-dimensional cube
\[C^\mon=\left\{(a_1,\ldots, a_d) : 0\le a_i \le r_i \text{ for all } 1\le i \le d\right\}\]
which contains $Y^\mon$. Let $K=|C^\mon|=r_1\ldots r_d$ denote the number of boxes in $C^\mon$.
The reverse lexicographic order 
\[C^\mon_\RL=(\nu_1,\ldots, \nu_{|C^\mon|})\]
of the boxes of $C^\mon$ then defines the toric sequence $\pi^+$ associated to $\pi^\mon$ as follows.  
For $1\le t \le K-1$ we define
\[\pi_t^+ = \begin{cases} [t] & \text{ if } \nu_t \in C^\mon \setminus Y^\mon\\
\tau_t & \text{ if } \nu_t \in  Y^\mon. 
\end{cases}\]
Then by construction  
\begin{enumerate}
\item $\pi^+$ is a toric sequence with respect to the extended flag $\ff^+=(e_1,\ldots, e_{K-1})$, and $\mon^+=\mon^{\pi^+}\in \Hilb^{K}(\CC^{K-1})$. 
\item Let $\underline{\mon} \subset \CC[x_1,\ldots, x_{K-1}]$ be the monomial ideal corresponding to the complete set $\{\tau_t^+:\nu_t \in Y^\mon\}$. Then $\CC[x_1,\ldots, x_{K-1}]/\underline{\mon} \simeq \CC[x_1,\ldots, x_n]/\mon$. 
\item $A^+=\CC[x_1,\ldots, x_{K-1}]/\mon^+$ is the socle extension of $A=\CC[x_1,\ldots x_{K-1}]/\underline{\mon}$ in the coordinate directions $\{t: \nu_t \in C^\mon \setminus Y^\mon\}$.
\end{enumerate}


\subsection{Proof of Theorem \ref{main1}}\label{subsec:proof}
Let $\mon \subset \CC[x_1,\ldots, x_n]$ be a monomial ideal and \[A=\CC[x_1,\ldots, x_n]/\mon \in \Hilb^{k+1}_0(\CC^n)\]
the corresponding algebra. Following \S \ref{toricExtensionSection}, we construct the $T$-monomial ideal $\mon^+ \subset \CC[x_1,\ldots, x_{K-1}]$ with respect to some extended flag $\ff^+=(e_1,\ldots, e_{K-1})$, which by Theorem \ref{main1toric} sits in the toric Hilbert scheme:  
\[A^+=\CC[x_1,\ldots, x_{K-1}]/\mon^+ \in \THilb^{K}_{\ff^+}(\CC^{K-1}).\]
where $K=|C^\mon|$.
$A^+$ is the (iterated) socle extension of an algebra isomorphic by $A$:
\[A^+=\epsilon^m(\underline{A}) \text{ with } \underline{A}=\CC[x_1,\ldots, x_{K-1}]/\underline{\mon}\simeq \CC[x_1,\ldots, x_n]/\mon=A\]
where $m=|C^\mon|-|Y^\mon|$. 
If $k\le n$, then by Theorem \ref{thm:descent} $A\simeq \underline{A} \in \CHilb^{k+1}_0(\CC^n)$. The following lemma shows that we can assume that $k\le n$ holds.
\begin{lemma}\label{lemma:dim} Let $V \subset W$ complex vector spaces and $\Hilb^{k}(V) \subset \Hilb^k(W)$ the corresponding Hilbert scheme of points. If $\xi \in \Hilb^{k+1}_0(V) \cap \CHilb^{k+1}_0(W)$ then $\xi \in \CHilb^{k+1}_0(V)$.  
\end{lemma}
\begin{proof} By Theorem \ref{bszmodel} 
\[\xi = \lim_{m\to \infty} \phi^\grass(v_1^m,\ldots, v_k^m)\]
for a regular sequence $(v_1^m,\ldots, v_k^m) \in J_k^\reg(1,W)$, that is, $v_i^m \in W$ and $v_1^m\neq 0$ for $1\le i\le k$. Fix an inner product on $W$ and let $\pi:W \to V$ denote the orthogonal projection. Since $\xi \in \Hilb^{k+1}_0(V)$, we have 
\[\xi=\pi(\xi)=\lim_{m\to \infty} \phi^\grass(\pi(v_1^m),\ldots, \pi(v_k^m))=[\pi(v_1^m) \wedge \ldots \wedge \Sigma_{i_1+\ldots i_s=k}\pi(v_{i_1}^m)\ldots \pi(v_{i_k}^m)]\]
and hence $\pi(v_1^m)\neq 0$ for $m\gg 0$, and by Theorem \ref{bszmodel} $\xi \in \CHilb^{k+1}_0(V)$. 
\end{proof}
We apply Lemma \ref{lemma:dim} with $V=\CC^n \subset W=\CC^N$ for $N\ge k$, to conclude that we can assume that $k$ does not exceed the dimension $n$ in the argument above. This completes the proof of Theorem \ref{main1}.
\section{Final remarks}\label{sec:finalremarks}

\subsection{On the characteristic of $\bk$} Although we set our algebraically closed field to $\bk=\CC$ in this paper, all arguments go through to fields with characteristic zero. For positive characteristic the picture is less shiny: one of the main issues is weather/how the test curve model behaves in characteristic $p$. In Theorem \ref{bszmodel} we identified the curvilinear Hilbert scheme as the Zariski closure of the image of the map $\phi^\grass$
\[\CHilb^{k+1}_0(\CC^n)=\overline{\mathrm{im}(\phi^\grass)} \subset \grass_k(\symdot).\]
which turns out to be the closure of the $\GL(n,\bk)$ orbit of
\[p_{k,n}=\phi(e_1,\ldots, e_k)=\mathrm{Span}_\CC (e_1,e_2+e_1^2,\ldots, \Sigma_{i_1+\ldots i_s=k}e_{i_1}\ldots e_{i_s})\] 
The point $p_{k,n}$ is well-defined in any positive characteristics $p$ (due to the presence of the terms $e_1^j$ in the $j$th wedge factor), but its exact form depends of $p$ when $p$ is small compared to $k$. Indeed, for $p=2,k=3$
\[p_{3,n}=e_1 \wedge (e_2+e_1^2) \wedge (e_3+2e_1e_2+e_1^3)=e_1 \wedge (e_2+e_1^2) \wedge (e_3+e_1^3)\]
the $e_1e_2$ term is missing. As a result the monomial ideal $[e_1 \wedge e_2 \wedge e_1e_2]=(x^2,y^2)$ is not visible for our model, and it is not in the curvilinear component. But our argument tells us that the monomial ideals in $\CHilb^k_0(\AAA_{\bk}^n)$ are exactly those which one can see in the test curve model.  

\subsection{Quot schemes} With the shorthand notation $S=\CC[x_1,\ldots, x_n]$, the Hilbert scheme of points on $\CC^n$ parametrises $k$ dimensional quotients of $S$. The Quot scheme is the higher rank version of this: it parametrises $k$-dimensional quotients of the free $S$-module of rank $r$:
\[\Quot_r^k(\CC^n)=\{J \subset S^{\oplus r}: \dim(S^{\oplus r}/J)=k\}\]
Classifying these modules to any significant extent is highly challenging, and the study of components, singularities, and the deformation theory of quot schemes present equally formidable obstacles, see \cite{joachimquot, joachimsurvey}.
$\Quot_r^k(\CC^n)$ is naturally endowed with a $T^r \times \GL(n)$-action induced from the $\GL(n)$ action on $\CC^n$ and the rescaling action on the $r$ components. The fixed points for the action of the torus $T^r \times T^n$ correspond to monomial quotients of the form $S/\mon_1 \oplus \ldots \oplus S/\mon_r$ where $\mon_1,\ldots, \mon_r$ are monomial ideals of $S$. The curvilinear component is the closure of the curvilinear locus:
\[\CQuot^k_r(\CC^n)=\overline{\{J \subset S^{\oplus r}: S^{\oplus r}/J \simeq \CC[t]/t^k\}}.\]
\begin{remark} Let $\pi_i:\Quot_r^k(\CC^n) \to \cup_{m=1}^k\Hilb^m(\CC^n)$ denote the projection to the $i$th factor for $i=1,\ldots r$. Quotients of Morin algebras are Morin, hence if $J\subset S^{\oplus r}$ is curvilinear, that is, $S^{\oplus r}/J \simeq \CC[t]/t^k$ then the quotient ideal $\pi_i(J) \subset S$ is curvilinear, that is, $S/\pi_i(J)\simeq \CC[t]/t^{k_i(J)}$ for some $k_i(J)\le k$. The curvilinear Quot scheme has the following characterisation:
\[S^{\oplus r}/J \simeq \CC[t]/t^k \Leftrightarrow k_i(J)=k \text{ for at least one } 1\le i \le r.\]
\end{remark}
\begin{theorem}
All torus fixed points of $\Quot_r^k(\CC^n)$ sit on the curvilinear component $\CQuot^k_r(\CC^n)$.
\end{theorem}
\begin{proof}
Torus fixed modules $\xi \in \Quot_r^k(\CC^n)$ are the monomial modules, which are direct sums of the form $S/\mon_1 \oplus \ldots \oplus S/\mon_r$, where $\mon_1,\ldots, \mon_r \subset S$ are monomial ideals. Applying Theorem \ref{main1} to the direct summands it follows that if $\xi$ is monomial, then 
\[\xi \in \overline{\{\xi_1 \oplus \ldots \oplus \xi_r: \xi_i \simeq \CC[t]/t^{k_i}\}}\] 
is a limit of curvilinear modules. So it is enough to show that a module whose components are curvilinear is a limit of curvilinear modules. To see this, by reparametrisation of the factors it is enough to prove that $\CC[x]/(x^{k_1}) \oplus \ldots \oplus \CC[x]/(x^{k_r})$ is a limit of curvilinear modules, where $k_1\ge k_2 \ge. \ldots \ge k_r$. By Theorem \ref{main1} the algebra $\CC[x,y]/(x^{k_1}, x^{k_2}y, \ldots ,x^{k_r}y^{r-1},y^r)$ is a limit of curvilinear ones. 
But $\CC[x]/(x^{k_1}) \oplus \ldots \oplus \CC[x]/(x^{k_r})$ is a limit of algebras isomorphic to $\CC[x,y]/(x^{k_1}, x^{k_2}y, \ldots ,x^{k_r}y^{r-1},y^r)$, where we make multiplication by $y$ degenerate to the zero action. 
\end{proof}

\subsection{Geometric subsets of Hilbert scheme of points}
Let $A_1,\ldots, A_s$ be finite dimensional quotient algebras of $\CC[x_1, \ldots, x_n]$ of dimension $\dim_\CC(A_i)=k_i$ with $k=k_1+\ldots +k_s$. The geometric subset
\[\Hilb^{A_1,\ldots, A_s}(\CC^n)=\overline{\{\xi=\xi_1 \sqcup \ldots \sqcup \xi_s: \calo_{\xi_i} \simeq A_i\}}\subset \Hilb^k(\CC^n)\]
is the closure of the $(A_1,\ldots, A_s)$-locus formed by sub-schemes supported on $s$ points with prescribed local algebra at each point. These geometric subsets play central role in several enumerative geometry problems and mathematical physics, see \cite{berczitau3}. When $A_i=\CC[x_1,\ldots,x_n]/\mon_i$ is monomial algebra for $1\le i \le s$, \cite{berczitau3} defines the sum $A=A_1+\ldots +A_s$ which is also monomial, and \cite{berczitau3} defines the curvilinear component of the geometric subset as 
\[\CHilb^{A_1,\ldots, A_s}(\CC^n)=\Hilb^{A_1+\ldots +A_s}_0(\CC^n),\]
which corresponds to the punctual component where the $s$ points of support collide at the origin along a smooth curve on $\CC^n$. In a forthcoming paper, using the methodology of the present paper,  we will tackle the following
\begin{conjecture}
All monomial ideals which sit in the geometric subset $\Hilb^{A_1,\ldots, A_s}(\CC^n)$ also sit in its curvilinear component $\Hilb^{A_1+\ldots +A_s}_0(\CC^n)$, and it is possible to characterise these monomial ideals using a test map model.
\end{conjecture}

\bibliographystyle{abbrv}
\bibliography{BercziSvendsen.bib}

\end{document}